\documentclass[a4paper,12pt]{article}

\usepackage{amsthm,mathrsfs,bbm,amssymb,here,amsmath,cite,textcomp,stmaryrd}

\usepackage{fullpage}

\usepackage{fancyhdr}
\usepackage{makeidx}
\usepackage{graphicx}
\usepackage{xr}
\usepackage[pdftex]{hyperref}
\usepackage{verbatim}

\theoremstyle{definition}
\newtheorem{theorem}{Theorem}[section]
\newtheorem{lemma}[theorem]{Lemma}
\newtheorem{proposition}[theorem]{Proposition}
\newtheorem{corollary}[theorem]{Corollary}

\newtheorem{definition}[theorem]{Definition}

\newtheorem{example}[theorem]{Example}

\newcommand{\mrmdd}{\mathrm{d}}
\newcommand{\perm}{\mathrm{Perm}}

\newcommand{\R}{\mathbb{R}}
\newcommand{\N}{\mathbb{N}}
\newcommand{\nQ}{\mathbb{Q}}

\newcommand{\E}{{\mathbf E}}
\renewcommand{\P}{{\mathbf P}}
\newcommand{\Q}{{\mathbf Q}}
\newcommand{\pR}{{\mathbf R}}
\newcommand{\one}{\mathbbm{1}}

\newcommand{\deq}{\stackrel{d}{=}}
\newcommand{\Z}{\mathbb{Z}}

\newcommand{\sB}{\mathcal{B}}

\newcommand{\sE}{\mathcal{E}}
\newcommand{\sM}{\mathcal{M}}
\newcommand{\sI}{\mathcal{I}}
\newcommand{\sS}{\mathcal{S}}
\newcommand{\sT}{\mathcal{T}}

\newcommand{\sF}{\mathcal{F}}
\newcommand{\sG}{\mathcal{G}}
\newcommand{\sA}{\mathcal{A}}
\newcommand{\sC}{\mathcal{C}}
\newcommand{\sD}{\mathcal{D}}

\newenvironment{benumber}{

\vspace{-.05in} \begin{enumerate}}{\end{enumerate}}

\begin{document}

\thispagestyle{empty}

\begin{center}
{\Large \bf Swap-invariant and exchangeable\\[.3em] random measures}\\

\vspace{.5cm}

{\large \scshape Felix~Nagel}\footnote{Email: felix.nagel@stat.unibe.ch}\\

\vspace{.2cm}

{\it Institute of Mathematical Statistics and Actuarial Science\\
University of Berne, Switzerland}

\vspace{.8cm}

\begin{minipage}{.85\textwidth}
\begin{center}
{\bf Abstract}
\end{center}
\vspace{-.3cm}
In this work we analyze the concept of swap-invariance, which is a weaker variant of exchangeability. A random vector $\xi$ in~$\R^n$ is called swap-invariant if $\,\E \,\big| \!\sum_j u_j \xi_j \big|\,$ is invariant under all permutations of $(\xi_1, \ldots, \xi_n)$ for each $u \in \R^n$. We extend this notion to random measures. For a swap-invariant random measure $\xi$ on a measure space $(S,\sS,\mu)$ the vector $(\xi(A_1), \ldots, \xi(A_n))$ is swap-invariant for all disjoint $A_j \in \sS$ with equal $\mu$-measure. Various characterizations of swap-invariant random measures and connections to exchangeable ones are established. We prove the ergodic theorem for swap-invariant random measures and derive a representation in terms of the ergodic limit and an exchangeable random measure. Moreover we show that diffuse swap-invariant random measures on a Borel space are trivial. As for random sequences two new representations are obtained using different ergodic limits.

\vspace{.2cm}

{\bf Keywords:} ergodic theorem; exchangeability; random measure; swap-invariance; zonoid equivalence.

\vspace{.2cm}

{\bf AMS MSC 2010:} 28D99; 37A50; 60F25; 60G09; 60G57.
\end{minipage}
\end{center}

\vspace{.5cm}

\section{Introduction}

Two integrable random vectors $\xi$ and $\eta$ in~$\R^n$ are called {\em zonoid equivalent} if $\E \left| \langle u, \xi \rangle \right| = \E \left| \langle u, \eta \rangle \right|$ for every $u \in \R^n$, where $\langle \,\cdot\,, \,\cdot\, \rangle$ denotes the Euclidean inner product. A vector~$\xi$ of integrable random variables is called {\em swap-invariant} if $\xi \circ \pi$ and $\xi$ are zonoid equivalent for all permutations $\pi$ of $\left\{ 1, \ldots, n \right\}$. Swap-invariance is weaker than exchangeability. Both exchangeability and swap-invariance are extended to random sequences by requiring the respective property for all finite-dimensional distributions. Swap-invariant sequences are introduced and analyzed in~\cite{mss:14}. An important property of a swap-invariant sequence $\xi$ is that $n^{-1} \sum_{j = 1}^n \xi_j \to X$ almost surely as $n \to \infty$ for some random variable~$X$ (cf.\  \protect{\cite[Theorem 17]{mss:14}}). In contrast to the ergodic theorem for integrable exchangeable sequences (see for example \protect{\cite[Theorem 10.6]{kallenberg:fou}}), this convergence is not necessarily in~$L^1$.

The definitions of exchangeability and swap-invariance can both be extended to random measures on a measure space $(S, \sS, \mu)$ where $\mu$ is a deterministic reference measure. Whereas exchangeable random measures are well known (see e.g.\ \cite{kallenberg:sym} and \protect{\cite[Chapter 10]{kallenberg:fou}}), our definition of swap-invariance is new, and it is strictly weaker than exchangeability. If a random measure $\xi$ on a measure space $(S,\sS,\mu)$ is $\mu$-exchangeable, then the vector $(\xi(A_1), \ldots, \xi(A_n))$ is exchangeable for all disjoint $A_j \in \sS$ with equal $\mu$-measure; if $\xi$ is $\mu$-swap-invariant, then this vector is merely swap-invariant. Exchangeable or swap-invariant sequences arise as special cases where $S = \N$ and $\mu$ is the counting measure.

We show the ergodic theorem for swap-invariant random measures, that is
\[
\frac{\xi(A_n)}{\mu(A_n)} \to X \quad\quad \mbox{a.s.\ as $n \to \infty$}
\]
for some integrable random variable~$X$. Here $(A_n)_{n \geq 1}$ is any increasing sequence of measurable sets such that $\mu(A_n) < \infty$ and $\mu(A_n) \to \infty$ as $n \to \infty$. The ergodic limit always exists if $\mu$ is atomless, and it is independent of the sequence of sets under additional assumptions. As an important consequence we obtain the representation
\begin{equation}
\label{intro-repr}
\xi = X \eta
\end{equation}
where the random measure $\eta$ is exchangeable under a certain probability measure.

In Section~\ref{sec-rm} we first give five characterizations of swap-invariant random measures and their counterparts in the exchangeable case. Then a construction method for swap-invariant random measures is provided, which is based on a change of the probability measure. We give an example of a swap-invariant non-exchangeable point process derived from a Poisson process. For finite $\mu$ we demonstrate that a swap-invariant random measure can be represented as $\xi = \xi(S) \, \eta$ where $\eta$ is exchangeable under a certain probability measure. This formula resembles~\eqref{intro-repr}; however in the case of finite $\mu$ no ergodic limit is involved. If the reference measure $\mu$ is $\sigma$-finite and atomless, and if the intensity measure~$\E \, \xi$ is $\sigma$-finite, then $\E \, \xi = c \mu$ for some $c \geq 0$. Finally it is shown that a swap-invariant diffuse random measure on a Borel space has the form $\alpha \mu$ where $\alpha \geq 0$ is a random variable. This fact is known for exchangeable random measures.

Section~\ref{sec-sisequ} of this paper is devoted to a detailed analysis of the connection between swap-invariant and exchangeable sequences. These results are the basis for the proof of~\eqref{intro-repr} and are interesting by themselves. First it is shown that each swap-invariant sequence of random variables with only two positive values is exchangeable. We then present a simple construction method for swap-invariant sequences, which consists in a multiplication of a given swap-invariant or exchangeable sequence by a random factor and a simultaneous change of the probability measure. In the remainder of Section~\ref{sec-sisequ} we prove that two large classes of swap-invariant sequences can always be represented in this way, namely by using an ergodic limit as the random factor; in one case we use the ergodic limit of the sequence itself, in the other case the limit of $p$-norms of the means.

Finally in Section~\ref{sec-erg} the ergodic theorem for swap-invariant random measures is shown and \eqref{intro-repr} is derived. We also prove a variant of this theorem for the special case of exchangeable random measures.

\section{Swap-invariant random measures}
\label{sec-rm}

\subsection{Preliminaries}

Let $(\Omega, \sF, \P)$ be a probability space and $(S,\sS)$ a measurable space. A {\em random measure $\xi$ on~$S$} is a map $\xi: \Omega \times \sS \to \overline{\R}_+$ such that $\xi ( \,\cdot\, , M)$ is a $\overline{\R}_+$-valued random element for each $M \in \sS$, and $\xi (\omega, \,\cdot\, )$ is a measure on~$S$ for each $\omega \in \Omega$. The {\em intensity measure} of~$\xi$ is $\E \, \xi (M)$ for $M \in \sS$.

If $(A,\sA)$ and $(B,\sB)$ are measurable spaces, $\mu$~is a measure on~$A$, and $f : A \to B$ is measurable, then the image of $\mu$ under~$f$ is denoted by $\mu \circ f^{-1}$, which is a measure on~$B$. If $\xi$ is a random measure on~$A$, $\xi \circ f^{-1}$ is defined pointwise for each~$\omega$, and therefore is a random measure on~$B$.

A random measure $\xi$ on~$S$ is called {\em $\sigma$-finite} if  there exists a fixed measurable partition $(S_j)_{j \geq 1}$ of~$S$ such that $\xi (S_j) < \infty$ almost surely for every $j \geq 1$. An {\em atom} $M$ of a measure $\mu$ is a measurable set such that $\mu(M) > 0$ and for each measurable $C \subset M$ either $\mu(C) = 0$ or $\mu(C) = \mu(M)$ holds; $\mu$~is called {\em atomless} if it has no atoms. If $(A,\sA)$ and $(B,\sB)$ are measurable spaces, then $f : A \to B$ is called a {\em Borel isomorphism} if $f$ is a bijection and both $f$ and $f^{-1}$ are measurable. A measurable space $(S,\sS)$ is called a {\em Borel space} if $S$ is Borel isomorphic to a Borel subset of~$([0,1], \sB[0,1])$. A measure $\mu$ on a Borel space~$S$ is called {\em diffuse} if $\mu (\left\{s\right\}) = 0$ for each $s \in S$.

For $n \geq 1$ we write $\perm(n)$ for the family of permutations of~$\left\{1, \ldots, n \right\}$.

\subsection{Exchangeability}
\label{subsec-exch-rm}

The definition of exchangeable random measures relies on the following lemma.

\begin{lemma}
\label{exch-random}
Let $\xi$ be a random measure on a measure space $(S,\sS,\mu)$. Consider the following statements:
\begin{benumber}
\item \label{exch-random-twosets} $( \xi (A_1), \ldots, \xi (A_n) ) \deq ( \xi (B_1), \ldots, \xi (B_n) )$ for any disjoint $A_1\,, \ldots, A_n \in \sS$, $n \geq 1$, and any disjoint $B_1\,, \ldots, B_n \in \sS$ with $\mu (A_j) = \mu (B_j)$ for $1 \leq j \leq n$.
\item \label{exch-random-twofunctions} $\int f \, d\xi \,\deq\, \int g \, d\xi\,$ for each two measurable functions $f, g : S \to \R_+$ with $\mu \circ f ^{-1} = \mu \circ g^{-1}$.
\item \label{exch-random-twotrans} $\xi \circ f^{-1} \deq \xi \circ g^{-1}$ for each two measurable functions $f, g : S \to \R_+$ with $\mu \circ f ^{-1} = \mu \circ g^{-1}$.
\item \label{exch-random-onetrans} $\xi \circ f^{-1} \deq \xi\,$ for each measurable function $f : S \to S$ with $\mu \circ f^{-1} = \mu$.
\item \label{exch-random-perm} There exists a sequence $c_k \downarrow 0$ such that for each $k \geq 1$, $n \geq 1$, any disjoint $A_1\,, \ldots, A_n \in \sS$ with $\mu (A_j) = c_k$ for $1 \leq j \leq n$, and each $\pi \in \perm(n)$, we have
\begin{equation*}
\big( \xi (A_1), \ldots, \xi (A_n) \big) \,\deq\, \big( \xi (A_{\pi(1)}), \ldots, \xi (A_{\pi(n)})\big) \,.
\end{equation*}
\end{benumber}
Statements \eqref{exch-random-twosets}, \eqref{exch-random-twofunctions}, and \eqref{exch-random-twotrans} are equivalent, and they imply \eqref{exch-random-onetrans} and~\eqref{exch-random-perm}. If $\mu$ is atomless and $\sigma$-finite, \mbox{$\mu (S) = \infty$}, and $\xi$ is $\sigma$-finite, then \eqref{exch-random-perm} implies~\eqref{exch-random-twosets}. If, in addition, $(S,\sS)$~is Borel, then also \eqref{exch-random-onetrans} implies~\eqref{exch-random-twosets}.
\end{lemma}

\begin{definition}
\label{def-exch-rm}
Let $\xi$ be a random measure on a measure space $(S,\sS,\mu)$. If $\xi$ satisfies either of the three conditions \eqref{exch-random-twosets}, \eqref{exch-random-twofunctions}, \eqref{exch-random-twotrans} in Lemma~\ref{exch-random}, then it is called {\em $\mu$-exchangeable}.
\end{definition}

It is easily seen that the $\mu$-exchangeable random measures on~$\N$, where $\mu$ is the counting measure, are exactly the exchangeable sequences in~$\overline{\R}_+$\,.

For the proof of Lemma~\ref{exch-random} some technical prerequisites are required. We need the following version of the Cram\'{e}r-Wold device for $\overline{\R}_+^{\, n}$\,. To this end the definition of the Euclidean inner product $\left\langle \,\cdot\, , \,\cdot\, \right\rangle$ is extended to~$\overline{\R}_+^{\, n} \!\times \overline{\R}_+^{\, n}$ in the canonical way.

\begin{lemma}
\label{cramwoldinf}
Two random elements $\xi$ and $\eta$ in~$\overline{\R}_+^{\, n}$ are equal in distribution if and only if $\left\langle u, \xi \right\rangle \deq \left\langle u, \eta \right\rangle$ for all $u \in \R_+^n$\,.
\end{lemma}

\begin{corollary}
\label{ranmeas}
Let $\xi$ and $\eta$ be two random measures on a measurable space~$(S, \sS)$. Then $\xi \deq \eta$ under each of the following conditions:
\begin{benumber}
\item \label{ranmeas-dissets} $( \xi(A_1), \ldots, \xi(A_n) \big) \deq \big( \eta(A_1), \ldots, \eta(A_n) )$ for any disjoint $A_1\,, \ldots, A_n \in \sS$, $n \geq 1$.
\item \label{ranmeas-posfun} $\int f \, d\xi \, \deq \, \int f \, d\eta\,$ for each measurable function $f \geq 0$.
\end{benumber}
\end{corollary}

The following result is required in Lemma~\ref{exch-random} for the implication from \eqref{exch-random-perm} to~\eqref{exch-random-twosets} and for the corresponding implication in Lemma~\ref{swap-random}.

\begin{lemma}
\label{setapprox}
Let $(S,\sS,\mu)$ be an atomless measure space, $J \geq 1$, $A_j \in \sS$ with $0 < \mu(A_j) < \infty$ for $1 \leq j \leq J$, and $c_k \downarrow 0$ as $k \to \infty$. Then there exist measurable sequences $(A_j^n)_{n \geq 1}$ with $A_j^n \uparrow A_j$ as $n \to \infty$ for $1 \leq j \leq J$, and a subsequence $(c_{k(n)})_{n \geq 1}$ such that
\[
0 < \mu \!\left(A_j \!\setminus\! A_j^n \right) \leq c_{k(n)}\,, \quad\quad \mu \!\left(A_j^n\right) = m_{jn} \, c_{k(n)}
\]
for $1 \leq j \leq J$, $n \geq 1$, and some integers $m_{jn} \geq 0$\,.
\end{lemma}

\begin{proof}
For $n = 1$ let $k_1 = 1$ and choose $A_j^1 \subset A_j$ and $m_{j1} \in \Z_+$ for $1 \leq j \leq J$ such that $\mu \!\left(A_j^1\right) = m_{j1} \, c_1$ and $0 < \mu \!\left(A_j \!\setminus\! A_j^1 \right) \leq c_1$\,. Note that we may have $m_{j1} = 0$. Now assume that, for some $N \geq 1$, we have already found $A_j^n$, $m_{jn}$\,, and $k_n$ for $1 \leq j \leq J$, $1 \leq n \leq N$. Set
\[
k_{N+1} = \min \big\{ k \geq 1 \, ; \, \mu \!\left(A_j \!\setminus\! A_j^N \right) > c_k \; (1 \leq j \leq J) \big\}
\]
and choose $A_j^{N + 1}$ and $m_{j (N+1)}$ such that
\[
A_j^N \subset A_j^{N+1}, \quad \mu \!\left(A_j^{N+1}\right) = m_{j (N+1)} \, c_{k(N+1)}\,, \quad 0 < \mu \!\left(A_j \!\setminus\! A_j^{N+1} \right) \leq c_{k(N+1)}\,. \qedhere
\]
\end{proof}

In order to show that \eqref{exch-random-onetrans} implies~\eqref{exch-random-twosets} in Lemma~\ref{exch-random} and the corresponding implication in Lemma~\ref{swap-random} we make use of the following result, which can be derived e.g.\ using Proposition~9.1.11 in~\cite{bogachev:mea} and Lemma~3.22 in~\cite{kallenberg:fou}.

\begin{lemma}
\label{propborelm}
Let $(C,\sC,\mu)$ and $(D,\sD,\nu)$ be measure spaces where $\mu$ is atomless, $D$ is Borel, and $\mu(C) = \nu(D) \in (0,\infty)$. Then there is a measurable function $h : C \to D$ such that $\mu \circ h^{-1} = \nu$.
\end{lemma}

\begin{proof}[Proof of Lemma~\ref{exch-random}]
We first show that \eqref{exch-random-twosets} implies~\eqref{exch-random-twotrans}. Let $f$ and $g$ be as in~\eqref{exch-random-twotrans}. If $n \geq 1$ and $A_1\,, \ldots, A_n \in \sB(\R_+)$ are disjoint, then also $ f^{-1}(A_1), \ldots, f^{-1}(A_n)$ are disjoint, and $g^{-1}(A_1), \ldots, g^{-1}(A_n)$ are disjoint. Moreover $\mu (f^{-1} (A_j)) = \mu (g^{-1} (A_j))$ for $1 \leq j \leq n$. Hence by~\eqref{exch-random-twosets}
\[
\left( \xi (f^{-1} (A_1)), \ldots, \xi (f^{-1} (A_n)) \right) \,\deq\, \left( \xi (g^{-1} (A_1)), \ldots, \xi (g^{-1} (A_n)) \right) .
\]
By Corollary~\ref{ranmeas}~\eqref{ranmeas-dissets} the two random measures $\xi \circ f^{-1}$ and $\xi \circ g^{-1}$ on~$\R_+$ have the same distribution.

In order to show that \eqref{exch-random-twotrans} implies~\eqref{exch-random-twofunctions}, let $f$ and $g$ be as required and note that
\[
\int_S f \, d\xi = \int_{\R_+} \! x \; d \!\left( \xi \circ f^{-1}\right)\! (x)  \deq
\int_{\R_+} \! x \; d \!\left( \xi \circ g^{-1}\right)\! (x) = \int_S g \, d\xi \,.
\]

To see that \eqref{exch-random-twofunctions} implies~\eqref{exch-random-twosets}, let $n$, $(A_j)_{1 \leq j \leq n}$\,, and $(B_j)_{1 \leq j \leq n}$ be as in~\eqref{exch-random-twosets}, $u \in \R_+^n$\,, and define the functions \mbox{$f = \sum_{j = 1}^n u_j \one_{A(j)}$} and \mbox{$g = \sum_{j = 1}^n u_j \one_{B(j)}$}\,. We have, for $B \in \sB(\R_+)$,
\[
\mu \!\left\{f \in B \right\} = \sum_{j = 1}^n \mu(A_j) \one_B (u_j) = \sum_{j = 1}^n \mu(B_j) \one_B (u_j) = \mu \!\left\{g \in B \right\} .
\]
Hence by assumption
\[
\sum_{j = 1}^n u_j \, \xi(A_j) = \int f \, d\xi \,\deq\, \int g \, d\xi = \sum_{j = 1}^n u_j \, \xi(B_j) \,.
\]
Now \eqref{exch-random-twosets} follows by Lemma~\ref{cramwoldinf}.

The implications from \eqref{exch-random-twosets} to~\eqref{exch-random-onetrans} and from \eqref{exch-random-twosets} to \eqref{exch-random-perm} are rather straightforward to prove.

We next prove that \eqref{exch-random-twosets} follows from~\eqref{exch-random-perm} under the additional assumptions. Let $n$, $(A_j)_{1 \leq j \leq n}$\,, and $(B_j)_{1 \leq j \leq n}$ be as in~\eqref{exch-random-twosets}, and $(c_k)_{k \geq 1}$ as in~\eqref{exch-random-perm}. First assume that \mbox{$\mu (A_j) \in (0,\infty)$} for $1 \leq j \leq n$. We may additionally assume that \mbox{$A_i \cap B_j = \emptyset$} for $1 \leq i, j \leq n$. For otherwise we could choose a third family of measurable disjoint sets $(C_j)_{1 \leq j \leq n}$ such that $\mu (C_j) = \mu (A_j)$ for $1 \leq j \leq n$, and $C_j \cap A_i = \emptyset$ and $C_j \cap B_i = \emptyset$ for $1 \leq i, j \leq n$, and then conclude that
\[
\big(\xi (A_1), \ldots, \xi (A_n) \big) \,\deq\, \big(\xi (C_1), \ldots, \xi (C_n) \big) \,\deq\, \big(\xi (B_1), \ldots, \xi (B_n) \big) .
\]
By Lemma~\ref{setapprox} we may choose sequences of measurable sets $(A_j^N)_{N \geq 1}$ and $(B_j^N)_{N \geq 1}$ for $1 \leq j \leq n$, a subsequence $(c_{k(N)})_{N \geq 1}$ of $(c_k)$, and integers $m_{jN} \geq 0$ for $1 \leq j \leq n$, $N \geq 1$, such that $A_j^N \uparrow A_j$ and $B_j^N \uparrow B_j$ as $N \to \infty$ and
\begin{align*}
 \mu \!\left(A_j^N\right) & = m_{jN} \, c_{k(N)}\,,  & 0 & < \mu \!\left(A_j \!\setminus\! A_j^N \right) \leq c_{k(N)}\,,\\
\mu \!\left(B_j^N\right) & = m_{jN} \, c_{k(N)}\,, & 0 & < \mu \!\left(B_j \!\setminus\! B_j^N \right) \leq c_{k(N)} \,.
\end{align*}
For large~$N$ we have $m_{jN} \geq 1$ for all~$j$. For such $N$ we may partition the sets $A_j^N$ and $B_j^N$ each into $m_{jN}$ sets of $\mu$-measure $c_{k(N)}$\,, and therefore obtain by~\eqref{exch-random-perm}
\[
\left(\xi \!\left( A_1^N \right), \ldots, \xi \!\left( A_n^N \right) \right) \,\deq\, \left(\xi \!\left( B_1^N \right), \ldots, \xi \!\left( B_n^N \right) \right) .
\]
Since $\xi \!\left( A_j^N \right)\! \uparrow \xi \!\left( A_j \right)$ and $\xi \!\left( B_j^N \right)\! \uparrow \xi \!\left( B_j \right)$ almost surely as $N \to \infty$ for $1 \leq j \leq n$, we have
\[
\big(\xi (A_1), \ldots, \xi (A_n) \big) \,\deq\, \big(\xi (B_1), \ldots, \xi (B_n) \big) \,.
\]
So far we have assumed that $\mu (A_j) \in (0,\infty)$ for $1 \leq j \leq n$. We next relax this assumption and assume only that $\mu (A_j) \in \R_+$ for $1 \leq j \leq n$, i.e.\ some sets may have $\mu$-measure zero. It is sufficient to show that, for each $A \in \sS$, $\mu (A) = 0$  implies 
$\xi (A) = 0$ almost surely. Let $A \in \sS$ with $\mu (A) = 0$. By the $\sigma$-finiteness of~$\xi$ we may assume that $\xi (A) < \infty$ almost surely. If $A = \emptyset$, then clearly $\xi(A) = 0$. If $A \neq \emptyset$, choose measurable sequences $(A_k)_{k \geq 1}$ and $(B_k)_{k \geq 1}$ with $A_k \downarrow A$ and $B_k \downarrow \emptyset$ as $k \to \infty$, $\mu (A_k) = \mu (B_k) \in (0, \infty)$ for $k \geq 1$, and $\xi (A_1) < \infty$, $\xi (B_1) < \infty$ almost surely. The proof for positive $\mu$-measure implies that $\xi (A_k) \deq \xi (B_k)$ for $k \geq 1$. Since \mbox{$\xi (A_k) \downarrow \xi (A)$} and \mbox{$\xi (B_k) \downarrow 0$} almost surely as $k \to \infty$, we obtain $\xi (A) = 0$ almost surely. The general case $\mu (A_j) \in \overline{\R}_+$ for $1 \leq j \leq n$ is now easily obtained by approximation with sets of finite measure.

Finally we show the implication from \eqref{exch-random-onetrans} to~\eqref{exch-random-twosets} under the stated conditions. So let $n$, $(A_j)_{1 \leq j \leq n}$\,, and $(B_j)_{1 \leq j \leq n}$ be as in~\eqref{exch-random-twosets}. As for the preceding implication we first consider the case $\mu (A_j) \in (0,\infty)$ for $1 \leq j \leq n$. Define \mbox{$A = \bigcup_{j = 1}^n A_j$} and \mbox{$B = \bigcup_{j = 1}^n B_j$} and choose a set $F \in \sS$ disjoint from $A$ and $B$ with $\mu (F) \in (0,\infty)$. Set \mbox{$E = A \cup B \cup F$}, \mbox{$A_0 = E \!\setminus\! A$}, and \mbox{$B_0 = E \!\setminus\! B$}. It follows that $\mu (E) \in (0,\infty)$ and
\[
\mu (A_0) = \mu (E) - \mu (A) = \mu (E) - \mu(B) = \mu (B_0) < \infty \,.
\]
Moreover, since $F \subset A_0$\,, we have $\mu (A_0) = \mu (B_0) \in (0,\infty)$. Now note that, by Lemma~\ref{propborelm}, for any $C, D \in \sS$ with $\mu (C) = \mu (D) \in (0,\infty)$, there exists a measurable function \mbox{$h : C \to D$} such that $\mu (h^{-1} (M)) = \mu (M)$ for each $M \in \sS \cap D$. Applying this result to $A_j$ and $B_j$ for $j \in \left\{ 0, 1, \ldots, n \right\}$, we can choose respective functions $h_j : A_j \to B_j$\,. As $(A_j)_{0 \leq j \leq n}$ together with $E^c$ form a partition of~$S$, the following function is well-defined:
\[
f : S \to S, \quad\quad f = \left\{ \begin{array}{ll} h_j & \quad \mathrm{on} \; A_j \; (0 \leq j \leq n) \\[1em]
 \mathrm{id} & \quad \mathrm{on} \; E^c \end{array}
\right. \,.
\]
Since $(B_j)_{0 \leq j \leq n}$ together with $E^c$ also form a partition of~$S$, we have, for $M \in \sS$,
\[
f^{-1} (M) = \left( M \cap E^c \right) \,\cup\, \bigcup_{j = 0}^n h_j^{-1} \left( M \cap B_j \right) \,.
\] 
Hence $f$ is measurable. Moreover we compute $\mu (f^{-1} (M)) = \mu (M)$. By the assumption we therefore find that $\xi \circ f^{-1} \deq \xi$. Consequently
\[
\big( \xi (A_1), \ldots, \xi (A_n) \big) = \big( \xi (f^{-1} (B_1)), \ldots, \xi (f^{-1} (B_n)) \big) \deq \big( \xi (B_1), \ldots, \xi (B_n) \big) \,.
\]
So far we have assumed that $\mu (A_j) \in (0,\infty)$ for $1 \leq j \leq n$. The extension to general $A_j \in \sB(\overline{\R}_+)$ is now the same as in the proof of the implication from \eqref{exch-random-perm} to~\eqref{exch-random-twosets}.
\end{proof}

\subsection{Swap-invariance}
\label{subsec-swapmeas}

We now investigate random measures that have a weaker property than $\mu$-exchangeability, namely $\mu$-swap-invariance. As in the case of random sequences our definition is based on that of zonoid equivalence.

\begin{definition}
\label{def-zon-equ-rm}
Two random measures $\xi$ and $\eta$ on a measurable space $(S,\sS)$ are called {\em zonoid equivalent} if $\E \,\big| \sum_{j = 1}^n u_j \, \xi (A_j) \big| = \E \,\big| \sum_{j = 1}^n u_j \, \eta (A_j) \big|$ for all $n \geq 1$, $u \in \R^n$, and any disjoint $A_1\,, \ldots, A_n \in \sS$ with \mbox{$\E \, \xi (A_j) < \infty$} and \mbox{$\E \, \eta (A_j) < \infty$} for $1 \leq j \leq n$. 
\end{definition}

Note that the integrands in Definition~\ref{def-zon-equ-rm} are undefined for those points~$\omega \in \Omega$ where two terms in the sum are infinite with opposite signs. However this can happen only with probability zero due to the integrability assumptions. Clearly, for $n \geq 1$, two finite zonoid equivalent random measures on~$\left\{ 1, \ldots, n \right\}$ are zonoid equivalent random vectors in~$\R^n_+$\,. Zonoid equivalence of two random measures can be characterized as follows.

\begin{lemma}
\label{zon-equiv}
Two random measures $\xi$ and $\eta$ on a measurable space $(S,\sS)$ are zonoid equivalent if and only if $\E \left| \int f \, d\xi \right| = \E \left| \int f \, d\eta \right|$ for each measurable function $f : S \to \R$ with $\E \int \left| f \right| d\xi < \infty$ and $\E \int \left| f \right| d\eta < \infty$.
\end{lemma}

\begin{proof}
Assume that $\xi$ and $\eta$ are zonoid equivalent. We first prove the implication for simple functions \mbox{$f = \sum_{j = 1}^n u_j \one_{A(j)}$} where $A_j \in \sS$ and $u_j \in \R$ for $1 \leq j \leq n$ such that $\E \int \left| f \right| d\xi < \infty$ and $\E \int \left| f \right| d\eta < \infty$. We may assume that the sets $(A_j)_{1 \leq j \leq n}$ are disjoint and $u_j \neq 0$ for $1 \leq j \leq n$. It follows that \mbox{$\E \, \xi (A_j) < \infty$} and \mbox{$\E \, \eta (A_j) < \infty$}, and therefore
\[
\E \left| \int f \, d\xi \right| = \E \, \bigg| \sum_{j = 1}^n u_j \, \xi (A_j) \bigg| = \E \, \bigg| \sum_{j = 1}^n u_j \, \eta (A_j) \bigg| = \E \left| \int f \, d\eta \right| \,.
\]
Now let $f : S \to \R$ be an arbitrary measurable function that satisfies the integrability conditions. Choose a sequence of simple functions~$f_n$ such that $f_n \to f$ and $|f_n| \uparrow |f|$ pointwise as $n \to \infty$. By dominated convergence we obtain \mbox{$\E \left| \int f_n \, d\xi \right| \to \E \left| \int f \, d\xi \right|$} as \mbox{$n \to \infty$}, and similarly for $\eta$.
\end{proof}

Before defining swap-invariance in Definition~\ref{def-si-rm} below, we formulate five statements about first absolute moments of random measures, that mirror their counterparts in the exchangeable case in Lemma~\ref{exch-random}. The first three statements are equivalent and serve as definition.

\begin{lemma}
\label{swap-random}
Let $\xi$ be a random measure on a measure space $(S,\sS,\mu)$. Consider the following statements:
\begin{benumber}
\item \label{swap-random-twosets} $\E \, \big| \sum_{j = 1}^n u_j \, \xi (A_j) \big| = \E \, \big| \sum_{j = 1}^n u_j \, \xi (B_j) \big|$ for all $n \geq 1$, $u \in \R^n$, disjoint $A_1\,, \ldots, A_n \in \sS$, and disjoint $B_1\,, \ldots, B_n \in \sS$ such that $\mu (A_j) = \mu (B_j)$, \mbox{$\E \, \xi (A_j) < \infty$}, and \mbox{$\E \, \xi (B_j) < \infty$} for $1 \leq j \leq n$.
\item \label{swap-random-twofunctions} $\E \left| \int f \, d\xi \right| = \E \left| \int g \, d\xi \right|$ for each two measurable functions $f, g : S \to \R$ such that $\mu \circ f^{-1} = \mu \circ g^{-1}$, $\E \int \left| f \right| d\xi < \infty$, and  $\E \int \left| g \right| d\xi < \infty$.
\item \label{swap-random-twotrans} The random measures $\xi \circ f^{-1}$ and $\xi \circ g^{-1}$ on~$\R$ are zonoid equivalent for each two measurable functions $f, g : S \to \R$ such that $\mu \circ f^{-1} = \mu \circ g^{-1}$, $\E \int \left| f \right| d\xi < \infty$, and  $\E \int \left| g \right| d\xi < \infty$.
\item \label{swap-random-onetrans} The random measures $\xi \circ f^{-1}$ and $\xi$ are zonoid equivalent for each measurable function $f : S \to S$ with $\mu \circ f^{-1} = \mu$.
\item \label{swap-random-perm} There exists a sequence $c_k \downarrow 0$ such that for each $k \geq 1$, $n \geq 1$, any disjoint $A_1\,, \ldots, A_n \in \sS$ with $\mu (A_j) = c_k$ and $\E \, \xi (A_j) < \infty$ for $1 \leq j \leq n$, and each $\pi \in \perm(n)$, we have $\E \, \big| \sum_{j = 1}^n u_j \, \xi \!\left(A_j\right) \!\big| = \E \, \big| \sum_{j = 1}^n u_j \, \xi \!\left(A_{\pi(j)}\right) \!\big|$ for all $u \in \R^n$.
\end{benumber}
Statements \eqref{swap-random-twosets}, \eqref{swap-random-twofunctions}, and \eqref{swap-random-twotrans} are equivalent, and they imply \eqref{swap-random-onetrans} and~\eqref{swap-random-perm}. If $\mu$ is atomless and $\sigma$-finite, the intensity measure $\E \, \xi$ is $\sigma$-finite, and $\mu (S) = \infty$, then \eqref{swap-random-perm} implies \eqref{swap-random-twosets}. If, in addition, $(S,\sS)$ is Borel, then also \eqref{swap-random-onetrans} implies \eqref{swap-random-twosets}.
\end{lemma}

\begin{proof}
We only show that \eqref{swap-random-twotrans} implies~\eqref{swap-random-twofunctions}, the proofs of the other implications being similar to Lemma~\ref{exch-random}. Assume that $f$ and $g$ satisfy the conditions of~\eqref{swap-random-twofunctions}. Then $\xi \circ f^{-1}$ and $\xi \circ g^{-1}$ are zonoid equivalent by~\eqref{swap-random-twotrans}. Moreover
\[
\E \int_\R \left\vert x \right\vert d \!\left(\xi \circ f^{-1}\right) \!(x) = \E \int_S \left\vert f \right\vert d\xi < \infty,
\]
and similarly for~$g$. Hence we may apply Lemma~\ref{zon-equiv} to the identity map on~$\R$ and obtain
\[
\E \left\vert \int_\R \! x \; d \!\left(\xi \circ f^{-1}\right) \!(x) \right\vert = \E \left\vert \int_\R \! x \; d \!\left(\xi \circ g^{-1}\right) \!(x) \right\vert,
\]
and therefore $\E \left| \int f \, d\xi \right| = \E \left| \int g \, d\xi \right|$.
\end{proof}

\begin{definition}
\label{def-si-rm}
Let $\xi$ be a random measure on a measure space $(S,\sS,\mu)$. If $\xi$ satisfies either of the properties \eqref{swap-random-twosets}, \eqref{swap-random-twofunctions}, \eqref{swap-random-twotrans} in Lemma~\ref{swap-random}, then $\xi$ is called {\em $\mu$-swap-invariant}.
\end{definition}

Clearly, every $\mu$-exchangeable random measure is $\mu$-swap-invariant. Moreover the $\sigma$-finite $\mu$-swap-invariant random measures on~$\N$, where $\mu$ is the counting measure, are exactly the non-negative swap-invariant sequences. Note that if $\xi$ is a $\mu$-swap-invariant random measure on~$S$, then, for each $A \in \sS$ with $\E \, \xi (A) < \infty$, $\mu (A) = 0$ implies $\xi (A) = 0$~almost surely. Note also that statement~\eqref{swap-random-twotrans} in Lemma~\ref{swap-random} is equivalent to the stronger statement where the restriction to functions satisfying $\E \int \left| f \right| d\xi < \infty$ and  $\E \int \left| g \right| d\xi < \infty$ is dropped, as can be seen from the proof. The same is not true for statement~\eqref{swap-random-twofunctions}.

\subsection{Swap-invariance vs.\ exchangeability}
\label{subsec-connections}

The close connection between the concepts of swap-invariance and exchangeability can already be seen from the similarity of the characterizations \eqref{swap-random-twosets} to \eqref{swap-random-perm} in Lemmas~\ref{exch-random} and~\ref{swap-random}. Two more results are now established. Proposition~\ref{rm-swap-from-ex} provides a construction of $\mu$-swap-invariant random measures from a given $\mu$-swap-invariant random measure. In particular, one can use a $\mu$-exchangeable random measure to construct a $\mu$-swap-invariant random measure that is not $\mu$-exchangeable. Afterwards we show in Theorem~\ref{swapmeas} how a swap-invariant random measure on a space of finite measure can be expressed through an exchangeable random measure. A similar representation of swap-invariant random measures on a space of infinite measure can only be proven in Section~\ref{sec-erg} because it is based on the ergodic theorem and on the representation result for sequences derived in Section~\ref{sec-sisequ}.

\begin{proposition}
\label{rm-swap-from-ex}
Let $(S,\sS,\mu)$ be a measure space, and $\eta$ a random measure on~$S$ that is $\mu$-swap-invariant under a probability measure~$\Q$. Further let $X$ be a random variable with $X > 0$ \, $\Q$-almost surely and $\E_\Q \!\left[ X^{-1} \right] < \infty$, and $\P$ another probability measure defined by
\[
\frac{\mrmdd\P}{\mrmdd\Q} = \frac{1}{X \, \E_\Q \!\left[ X^{-1} \right]} \,.
\]
Then the random measure $\xi$ defined by $\xi = X \eta$ is $\mu$-swap-invariant under~$\P$.
\end{proposition}

\begin{example}
\label{ex-swap-rm}
Let $(S,\sS)$ be a measurable space, $\mu$ an atomless $\sigma$-finite measure on~$S$, and $\eta$ a Poisson process on~$S$ with intensity measure~$\mu$ under a probability measure~$\Q$. It follows by Lemma~\ref{exch-random}~\eqref{exch-random-twosets} that $\eta$ is $\mu$-exchangeable under~$\Q$. Assuming $\mu(S) \geq 2$ we may choose two disjoint measurable sets $K$ and $L$ with \mbox{$\mu (K) = \mu (L) = 1$}. Define the random variable $X = 1 + \one \!\left\{ \eta (K) > 0 \right\}$ and a new probability measure $\P$ by
\[
\frac{\mrmdd\P}{\mrmdd\Q} = \frac{1}{c \, X}, \quad\quad c = \E_\Q \!\left[ X^{-1} \right] = \frac{1}{2} \left( 1 + e^{-1} \right) .
\]
By Proposition~\ref{rm-swap-from-ex} the point process $\xi = X \eta$ is $\mu$-swap-invariant under~$\P$. We calculate
\[
\P \big( \xi (L) = 0 \big) = \frac{1}{e} \neq \frac{2}{1 + e} = \P \big( \xi (K) = 0 \big) \,.
\]
This shows that $\xi$ is not $\mu$-exchangeable under~$\P$.
\end{example}

For the next theorem the following lemma is needed, which is a consequence of~\protect{\cite[Theorem 2]{mss:14}} and~\protect{\cite[Lemma 1.35]{kallenberg:fou}}.

\begin{lemma}
\label{lemma-norm-ze}
Let $\left\| \,\cdot\, \right\|$ be a norm on~$\R^d$ where $d \geq 1$, and $\xi_1\,$, $\xi_2$ two random vectors in $\R^d$ that are integrable under a probability measure~$\P$, and either both symmetric or both supported by~$\R^d_+$\,. For $i \in \left\{ 1, 2 \right\}$ assume $\E_\P \!\left\| \xi_i \right\| > 0$ and define a probability measure~$\Q_i$ by
\[
\frac{\mrmdd\Q_i}{\mrmdd\P} = \frac{\left\| \xi_i \right\|}{ \E_\P \!\left\| \xi_i \right\|}\,.
\]
If $\xi_1$ and $\xi_2$ are zonoid equivalent under~$\P$, then $\E_\P \!\left\| \xi_1 \right\|  = \E_\P \!\left\| \xi_2 \right\|$ and $\Q_1 \!\left(\xi_1 / \!\left\| \xi_1 \right\| \in A \right) = \Q_2 \!\left( \xi_2 / \!\left\| \xi_2 \right\| \in A \right)$ for all $A \in \sB(\R^d)$.
\end{lemma}

\begin{theorem}
\label{swapmeas}
Let $(S,\sS,\mu)$ be a space of finite measure and $\xi$ a random measure on~$S$ that is $\mu$-swap-invariant under a probability measure~$\P$ with $\E_\P \xi (S) \in (0,\infty)$. Then there exists a random measure~$\eta$ that is $\mu$-exchangeable under the probability measure $\Q$ defined by
\[
\frac{\mrmdd\Q}{\mrmdd\P} = \frac{\xi(S)}{\E_\P \xi (S)}
\]
such that $\xi = \xi(S) \, \eta$.
\end{theorem}

\begin{proof}
Let $n \geq 1$ and $A_j\,, B_j \in \sS$ with $\mu (A_j) = \mu (B_j)$ for $1 \leq j \leq n$ such that $(A_j)_{1 \leq j \leq n}$ are disjoint and $(B_j)_{1 \leq j \leq n}$ are disjoint. Define $A_0 = S \setminus  \bigcup_{j = 1}^n A_j$ and $B_0 = S \setminus \bigcup_{j = 1}^n B_j$\,. Clearly, $\mu (A_0) = \mu (B_0)$. Now we define random vectors $\xi_A$ and $\xi_B$ in $\R_+^{n + 1}$ by
\[
\xi_{Aj} = \xi (A_j) \; \one \!\left\{ \xi (A_j) < \infty \right\}, \quad\quad \xi_{Bj} = \xi (B_j) \; \one \!\left\{ \xi (B_j) < \infty \right\}
\] 
for $0 \leq j \leq n$. Here the indicator functions are $\P$-almost surely equal to~$1$. Note that $\xi_A$ and $\xi_B$ are zonoid equivalent under~$\P$. Moreover $\left\| \xi_A \right\|_1 = \xi(S) =  \left\| \xi_B \right\|_1$ $\P$-almost surely where $\left\| x \right\|_1 = \sum_{j = 0}^n |x_j|$ for $x \in \R^{n + 1}$. Define
\[
\eta = \left\{ \begin{array}{ll} \displaystyle \frac{\xi}{\xi (S)} & \quad \mathrm{on} \;\left\{ \xi(S) > 0 \right\}\\[1em]
 0 & \quad \mathrm{on} \;\left\{ \xi(S) = 0 \right\} \end{array}
\right. .
\]
Applying Lemma~\ref{lemma-norm-ze},  it follows that, under~$\Q$,
\[
\big( \eta (A_1), \ldots, \eta (A_n) \big) \,\deq\, \big( \eta (B_1), \ldots, \eta (B_n) \big)
\]
Hence $\eta$ is $\mu$-exchangeable under~$\Q$.
\end{proof}

Note that in the special case if $\xi$ is a random probability measure, i.e.\ \mbox{$\xi (S) = 1$}, which is $\mu$-swap-invariant under~$\P$, Theorem~\ref{swapmeas} says that $\xi$ is even $\mu$-exchangeable under~$\P$. Theorem~\ref{swapmeas} is also needed in the proof of Theorem~\ref{theo-diffuse-rm}, that characterizes diffuse swap-invariant random measures.

\subsection{Intensity measure}
\label{subsec-int}

If $\mu$ is atomless, a simple consequence of $\mu$-swap-invariance (and therefore also of $\mu$-exchangeability) is that the intensity measure is proportional to~$\mu$.

\begin{theorem}
\label{siintmeas}
Let $(S,\sS, \mu)$ be a measure space where $\mu$ is atomless and $\sigma$-finite, and $\xi$ a $\mu$-swap-invariant random measure on~$S$ such that $\E \, \xi$ is $\sigma$-finite. Then $\E \, \xi = c \mu$ for some $c \in \R_+$\,.
\end{theorem}

\begin{proof}
First note that $\mu(S) = 0$ implies $\E \, \xi (S) = 0$ by the $\sigma$-finiteness of~$\E \, \xi$ and~$\mu$-swap-invariance.

Now assume that $\mu (S) \in (0,\infty)$ and $\E \, \xi (S) < \infty$. There is a function \mbox{$h : [0, \mu(S)] \to \R_+$} such that $\E \, \xi (A) = h (\mu(A))$ for each $A \in \sS$. Since $\mu$ is atomless, $h$ is uniquely defined on~$[0, \mu(S)]$. We next show that $h$ is additive. Let $x, y \in \R_+$ such that $x + y \in [0, \mu(S)]$. Choose sets $A, B \in \sS$ with $\mu (A) = x$, $\mu (B) = y$, and $A \cap B = \emptyset$. Then
\begin{eqnarray*}
\label{silin}
h(x + y) & \!\!=\!\! & h\big((\mu(A) + \mu(B)\big) \,=\, h \big(\mu(A \cup B) \big) \,=\, \E \, \xi (A \cup B)\\
 & \!\!=\!\! & \E \, \xi(A) + \E \, \xi(B) \,=\, h (\mu(A)) +  h (\mu(B)) \,=\, h(x) + h(y) \,.
\end{eqnarray*}
Consequently for $x, y \in \R_+$ with $0 \leq x < y \leq \mu(S)$, we have $h(x) \leq h(y)$. Moreover $h(0) = h \!\left(\mu(\emptyset)\right) = \E \, \xi(\emptyset) = 0$. Now we obtain for $p, q \in \N$ with $1 \leq p \leq q$:
\[
h \!\left(\frac{p}{q} \, \mu(S) \right) = p \, h \!\left(\frac{1}{q} \, \mu(S) \right) = \frac{p}{q} \, q \, h \!\left(\frac{1}{q} \, \mu(S) \right) = \frac{p}{q} \, h \!\left(\mu(S) \right),
\]
that is, $h(z \mu(S)) = z \, h(\mu(S))$ for $z \in \nQ \cap [0,1]$. Since $h$ is non-decreasing, this implies that $h$ is continuous at~$0$. As $h$ is additive, it follows that $h$ is continuous on~$[0, \mu(S)]$, and therefore $h(z \mu(S)) = z \, h(\mu(S))$ for $z \in [0,1]$. We obtain, for each $A \in \sS$,
\[
\E \, \xi(A) = \frac{\E \, \xi(S)}{\mu(S)} \, \mu(A) \,.
\]

Now consider the case that $\mu (S)$ or $\E \,\xi (S)$ or both quantities are infinite. Choose a measurable partition $( S_j)_{j \geq 1}$ of~$S$ such that $\mu (S_j) < \infty$ and $\E \,\xi (S_j) < \infty$ for all $j \geq 1$. Define \mbox{$J = \left\{ j \in \N \, ; \, \mu(S_j) > 0 \right\}$}. Clearly $\xi (S_j) = 0$ almost surely for $j \in \N \setminus J$. For each $j \in J$ define
\[
\mu_j ( \, \cdot \,) = \mu (S_j \cap \, \cdot \, ), \quad \quad \xi_j (\, \cdot \,)  = \xi (S_j \cap \, \cdot \, ),
\]
so that $\mu_j$ is a finite measure on~$S$ and $\xi_j$ is a random measure on~$S$ with finite intensity measure. Now fix $j \in J$. Note that $\xi_j$ is $\mu_j$-swap-invariant and $\mu_j$ is atomless. Hence it follows by the first part of the proof for finite measure space that
\[
\E \, \xi_j = \frac{\E \, \xi_j(S)}{\mu_j(S)} \, \mu_j \,.
\]
If $|J| \geq 2$, let $i, j \in J$ with $i \neq j$ and consider the partition
\[
\sM = \big\{ S_k \, ; \, k \geq 1,\, k \neq i,\, k \neq j \big\} \cup \left\{ S_i \cup S_j \right\}
\]
of~$S$. We define
\[
\mu_{ij} ( \, \cdot \,) = \mu \big( (S_i \cup S_j) \cap \, \cdot \,\big) , \quad \quad \xi_{ij} ( \, \cdot \,) = \xi \big( (S_i \cup S_j) \cap \, \cdot \,\big) \,.
\]
Applying the previous result to the new partition~$\sM$ we obtain, in particular,
\[
\E \, \xi (S_i) = \E \, \xi_{ij} (S_i) = \frac{\E \, \xi_{ij} (S)}{\mu_{ij} (S)} \, \mu_{ij} (S_i) = \frac{\E \, \xi_{ij} (S)}{\mu_{ij} (S)} \, \mu (S_i), \quad\quad \E \, \xi (S_j) = \frac{\E \, \xi_{ij} (S)}{\mu_{ij} (S)} \, \mu (S_j) \,.
\]
Thus $c = \E \, \xi (S_k) / \mu(S_k)$ for some $c \in \R_+$ and all $k \in J$, and therefore
\[
\E \, \xi = \sum_{j \in J} \E \, \xi_j = c \sum_{j \in J} \mu_j = c \mu \,. \qedhere
\]
\end{proof}

In the $\mu$-exchangeable case the assumptions in Theorem~\ref{siintmeas} can be slightly weakened as follows.

\begin{corollary}
\label{cor-int-exch}
Let $(S,\sS, \mu)$ be a measure space where $\mu$ is atomless and $\sigma$-finite, and $\xi$ a $\mu$-exchangeable random measure on~$S$ such that there exists $A \in \sS$ with $\mu (A) \in (0,\infty)$ and $\E \, \xi (A) < \infty$. Then $\E \, \xi = c \mu$ for some $c \in \R_+$\,.
\end{corollary}

\begin{proof}
We may choose a measurable finite cover or countable partition $(S_j)_{j \geq 1}$ of~$S$ such that $\mu (S_j) = \mu (A)$ for all $j \geq 1$. Then $\E \, \xi (S_j) = \E \, \xi (A)$ for $j \geq 1$, which shows the $\sigma$-finiteness of~$\E \, \xi$. Hence the conditions of Theorem~\ref{siintmeas} are satisfied.
\end{proof}

\begin{example}
\label{eq-swap-rm-int}
In Example~\ref{ex-swap-rm} we have $\E_\Q \eta = \mu$, which implies $\E_\P \xi = c^{-1} \mu$.
\end{example}

\subsection{Borel space}
\label{subs-Borel}

If $S$ is Borel and $\mu$ is atomless, Theorem~\ref{theo-diffuse-rm} below characterizes diffuse $\mu$-swap-invariant random measures. It is based on the following proposition which is a direct consequence of Proposition~1.22 in~\cite{kallenberg:sym}. It is formulated here for a general Borel space.

\begin{proposition}
\label{propdiff}
Let $(S,\sS)$ be a Borel space, $\hat{\mu}$ an atomless probability measure on~$S$, and $\eta$ a random measure on~$S$ that is $\sigma$-finite, almost surely diffuse, and $\hat{\mu}$-exchangeable. Then $\eta = a \hat{\mu}$ almost surely for some random variable $a \geq 0$.
\end{proposition}

\begin{proof}
Since $S$ is Borel, there is a Borel isomorphism $h : S \to [0,1]$. Let $\nu = \hat{\mu} \circ h^{-1}$. Clearly $\nu$ is an atomless probability measure on~$[0,1]$. It is known (see e.g.\ Theorem~2.1 in~\cite{walters:erg}) that the probability spaces $\left([0,1], \sB[0,1], \nu \right)$ and $\left([0,1], \sB[0,1], \lambda \right)$ where $\lambda$ denotes the Lebesgue measure are isomorphic in the sense that there are $L, M \in \sB[0,1]$ with $\nu(L) = 1$ and $\lambda(M) = 1$ and a Borel isomorphism $f: L \to M$ such that $\nu \circ f^{-1} = \lambda$. Define $K = h^{-1}(L)$ and the map $g: K \to M$, $g = f \circ (h |K)$. Then $g$ is a Borel isomorphism and
\[
\left( \hat{\mu} \circ g^{-1} \right) (M \cap \,\cdot\,) = \lambda ( \,\cdot\, ),\quad\quad
\left( \lambda \circ g \right) (K \cap \,\cdot\, ) = \hat{\mu}  ( \,\cdot\, ) \,.
\]

Now define the random measure $\zeta ( \,\cdot\, ) = \left( \eta \circ g^{-1} \right) (M \cap \,\cdot\, )$ on~$[0,1]$. We show that $\zeta$ is $\lambda$-exchangeable. Let $n \geq 1$ and $A_j\,, B_j \in \sB[0,1]$ with $\lambda (A_j) = \lambda (B_j)$ for $1 \leq j \leq n$ such that $(A_j)_{1 \leq j \leq n}$ are disjoint and $(B_j)_{1 \leq j \leq n}$ are disjoint. Then, for $1 \leq j \leq n$, 
\[
\hat{\mu} \!\left( g^{-1}(A_j \cap M) \right) = \lambda (A_j) = \lambda (B_j) = \hat{\mu} \!\left( g^{-1}(B_j \cap M) \right).
\]
The sets $(g^{-1}(A_j \cap M))_{1 \leq j \leq n}$ are disjoint, and the sets $(g^{-1}(B_j \cap M))_{1 \leq j \leq n}$ are disjoint as well. By the $\hat{\mu}$-exchangeability of~$\eta$,
\[
\left( \eta \!\left( g^{-1} (A_1 \cap M) \right), \ldots, \eta \!\left( g^{-1} (A_n \cap M) \right) \right) \,\deq\, \left( \eta \!\left( g^{-1} (B_1 \cap M) \right), \ldots, \eta \!\left( g^{-1} (B_n \cap M) \right) \right),
\]
and therefore
\[
\big( \zeta (A_1), \ldots, \zeta (A_n) \big) \,\deq\, \big( \zeta (B_1), \ldots, \zeta (B_n) \big) \,.
\]
This shows that $\zeta$ is $\lambda$-exchangeable. Clearly $\zeta$ is $\sigma$-finite. Moreover $\eta ( \,\cdot\, ) = \left( \zeta \circ g \right) (K \cap \, \cdot\, )$ almost surely. This implies, in particular, that $\zeta$ is almost surely diffuse. By Proposition~1.22 in~\cite{kallenberg:sym}, $\zeta = a \lambda$ almost surely for some random variable $a \geq 0$. Hence,
\[
\eta ( \,\cdot\, ) \,=\, \left( \zeta \circ g \right) (K \cap \,\cdot\, ) \,=\, a \left( \lambda \circ g \right) (K \cap \,\cdot\, ) \,=\, a \, \hat{\mu} ( \,\cdot\, ) \quad \mbox{a.s.} \qedhere
\]
\end{proof}

\begin{theorem}
\label{theo-diffuse-rm}
Let $(S,\sS)$ be a Borel space, $\mu$~an atomless and $\sigma$-finite measure on~$S$, and $\xi$ an almost surely diffuse random measure on~$S$ such that $\E \, \xi$ is $\sigma$-finite. Then $\xi$ is $\mu$-swap-invariant if and only if $\xi = \alpha \mu$ almost surely for some $\R_+$-valued integrable random variable~$\alpha$.
\end{theorem}

\begin{proof}
It is easy to see that $\xi$ is $\mu$-swap-invariant if it has the stated form, so we prove only the converse implication.

First note that, if $\xi$ is $\mu$-swap-invariant and $\mu(S) = 0$, then $\xi(S) = 0$ almost surely by Theorem~\ref{siintmeas}.

Next assume that $\mu (S) \in (0,\infty)$ and define the measure $\hat{\mu} = \mu / \mu(S)$. We denote the basic probability measure by~$\P$. We may assume that $\E_\P \xi(S) \in (0,\infty)$. Now define $\Q$ and $\eta$ as in Theorem~\ref{swapmeas}. By this theorem $\eta$ is $\hat{\mu}$-exchangeable under~$\Q$. It follows from the definition of~$\eta$ that it is $\Q$-almost surely diffuse. Hence we can apply Proposition~\ref{propdiff} and obtain $\eta = a \hat{\mu}\,$ $\Q$-almost surely for some random variable $a \geq 0$. Since $\eta(S) = 1 \,$ $\Q$-almost surely, we may set $a = 1$. Therefore $\xi = \alpha \mu \,$ $\P$-almost surely where $\alpha = \xi(S) / \mu(S)$.

Now assume that $\mu(S) = \infty$. Choose a measurable partition $(S_j)_{j \geq 1}$ of~$S$ such that $\mu (S_j) \in (0,\infty)$. By Theorem~\ref{siintmeas} we may assume that $\E \,\xi (S_j) \in (0,\infty)$ for all $j \geq 1$. Define, for $j \geq 1$,
\[
\mu_j ( \, \cdot \,) = \mu (S_j \cap \, \cdot \, ), \quad \quad \xi_j (\, \cdot \,)  = \xi (S_j \cap \, \cdot \, ) \,.
\]
Fix $j \geq 1$. Note that $\xi_j$ is $\mu_j$-swap-invariant. Moreover $\mu_j$ is atomless and $\xi_j$ is almost surely diffuse. Hence it follows by the first part of the proof for finite measure space that
\[
\xi_j = \frac{\xi_j(S)}{\mu_j(S)} \mu_j \quad \mbox{a.s.}
\]
Now fix $i, j \geq 1$ with $i \neq j$ and consider the measurable partition
\[
\sM = \big\{ S_k \, ; \, k \geq 1,\, k \neq i,\, k \neq j \big\} \cup \left\{ S_i \cup S_j \right\}
\]
of~$S$. Define
\[
\mu_{ij} ( \, \cdot \,) = \mu \big( (S_i \cup S_j) \cap \, \cdot \,\big) , \quad \quad \xi_{ij} ( \, \cdot \,) = \xi \big( (S_i \cup S_j) \cap \, \cdot \,\big) \,.
\]
Applying the previous result to the new partition~$\sM$ we obtain, in particular,
\[
\xi (S_i) = \xi_{ij} (S_i) = \frac{\xi_{ij} (S)}{\mu_{ij} (S)} \, \mu_{ij} (S_i) = \frac{\xi_{ij} (S)}{\mu_{ij} (S)} \, \mu (S_i), \quad\quad \xi (S_j) = \frac{\xi_{ij} (S)}{\mu_{ij} (S)} \, \mu (S_j) \quad \mbox{a.s.}
\]
Thus $\alpha = \xi (S_k) / \mu(S_k)$ almost surely for some random variable $\alpha \geq 0$ and all $k \geq 1$, and therefore
\[
\xi = \sum_{j \geq 1} \xi_j = \alpha \sum_{j \geq 1} \mu_j = \alpha \mu \quad \mbox{a.s.}
\]
Since $\E \,\xi (S_1) < \infty$, $\alpha$ is integrable.
\end{proof}

\section{Swap-invariant sequences}
\label{sec-sisequ}

\subsection{Sequences with finitely many values}

\begin{theorem}
\label{lemma-swap-fin}
Let $\xi$ be a swap-invariant random vector whose coordinates take only values $a$ and~$b$ with $\left| a \right| \neq \left| b \right|$. Then $\xi$ is exchangeable.
\end{theorem}

The following lemma is needed to cover the case of negative values.

\begin{lemma}
\label{zon-equ-abs}
Let $d \geq 1$ and $\xi$ and $\xi^\ast$ be integrable random vectors in~$\R^d$. If $\xi$ and $\xi^\ast$ are zonoid equivalent, then also the vectors $(|\xi_1|, \ldots, |\xi_d|)$ and $(|\xi_1^\ast|, \ldots, |\xi_d^\ast|)$ are zonoid equivalent.
\end{lemma}

\begin{proof}
Note that for each $u \in \R^d$ the function $f(x) = \big| \sum_{j = 1}^d u_j |x_j| \big|$ on~$\R^d$ is measurable, even, and positively homogeneous. By~\cite[Theorem 2]{mss:14} we obtain $\E f(\xi) = \E f(\xi^\ast)$ for such~$f$, which proves the assertion.
\end{proof}

\begin{proof}[Proof of Theorem~\ref{lemma-swap-fin}]
First assume that $0 \leq a < b$. Fix $n \geq 2$. For $1 \leq d \leq n$ we denote by $N_d$ the set of vectors $k \in \left\{ 1, \ldots, n \right\}^d$ such that $1 \leq k_1 < k_2 < \ldots < k_d$\,. Furthermore we define $N = \bigcup_{1 \leq d \leq n} N_d$\,. For $k \in N$ the dimension of the vector $k$ is denoted by~$|k|$. For a random vector~$\xi$ in~$\left\{ a, b \right\}^n$ define
\[
z_\xi (k) = \E \max \big\{ \xi_{k_1}, \ldots, \xi_{k_{|k|}} \big\}, \; k \in N\,.
\]
We first prove that the distribution of~$\xi$ is uniquely determined by~$z_\xi$\, and then show that $z_\xi$ is invariant under permutations of~$\xi$ if $\xi$ is swap-invariant. The marginal distributions of~$\xi$ are denoted as follows:
\[
p (k ; m) = \P \!\left( \xi_{k_1} = m_1, \ldots, \xi_{k_d} = m_d \right)
\]
where $1 \leq d \leq n$, $k \in N_d$\,, and $m \in \left\{ a, b \right\}^d$. Then, for $k \in N$,
\begin{equation}
\label{equ-prob-fin}
p (k; (a, \ldots, a)) = \frac{b - z_\xi (k)}{b - a}\,.
\end{equation}
We show that all marginal probabilities are functions of~$z_\xi$\,. This is obvious for $k \in N_1$ since \eqref{equ-prob-fin} gives
\[
p (k_1; a) = \frac{b - z_\xi (k_1)}{b - a}, \quad \mbox{and} \quad  p (k_1; b) = 1 - p (k_1; a)\,.
\]
Now let $1 \leq d \leq n-1$ and suppose that the probabilities $p (k; m)$ are known for all $k \in N_d$ and $m \in \left\{ a, b \right\}^d$. Fix $k \in N_{d + 1}$\,. For $1 \leq j \leq d + 1$ and $m \in \left\{ a, b \right\}^{d + 1}$, we obtain by summation over the $j$th coordinate that
\begin{align}
\label{equ-fin-it}
& p \,\big( \left( k_1, \ldots, k_{j-1}, k_{j+1}, \ldots, k_{d+1} \right) ; \left( m_1, \ldots, m_{j-1}, m_{j+1}, \ldots, m_{d+1} \right) \big) \\[.5em]
& = p \, \big( k ; \left( m_1, \ldots, m_{j-1}, a, m_{j+1}, \ldots, m_{d+1} \right) \big) + p \, \big(k ; \left( m_1, \ldots, m_{j-1}, b, m_{j+1}, \ldots, m_{d+1} \right) \big)\,. \nonumber
\end{align}
By \eqref{equ-prob-fin} and \eqref{equ-fin-it} all probabilities for the chosen~$k$ can be calculated iteratively. This shows that all marginal probabilities are determined by~$z_\xi$\,. Now let $\pi \in \perm(n)$. If $\xi$ is swap-invariant, then $z_\xi (k) = \E \max \left\{ \xi_1, \ldots, \xi_{|k|} \right\}$ for $k \in N$ by~\protect{\cite[Theorem 2]{mss:14}}, and therefore $z_\xi = z_{\xi \circ \pi}$\,. By the above argument $\xi$ and $\xi \circ \pi$ have the same distribution. Therefore $\xi$ is exchangeable.

To show the claim for general $a$ and $b$ note that the random vector $\left| \xi \right|$ is swap-invariant by Lemma~\ref{zon-equ-abs}, and hence it is exchangeable by the proof above. Since $\left| a \right| \neq \left| b \right|$, it follows that $\xi$ is exchangeable.
\end{proof}

It is not difficult to construct a random sequence in $\left\{ -1, +1 \right\}$ that is swap-invariant but not exchangeable, e.g.\ using the method of Proposition~\ref{swap-from-ex}. Moreover Example~\ref{ex-swap-from-exch-special} provides a swap-invariant but non-exchangeable sequence whose coordinates take three values. 

\subsection{Construction method}

The next proposition provides a method to construct swap-invariant sequences from another swap-invariant sequence. It is a direct consequence of the definition of swap-invariance.

\begin{proposition}
\label{swap-from-ex}
Let $\eta$ be a sequence that is swap-invariant under a probability measure~$\Q$. Further let $X$ be a random variable with $X \neq 0$ $\Q$-almost surely and $c = \E_\Q \!\left[ \, |X|^{-1} \right] < \infty$. Define another probability measure $\P$ by
\[
\frac{\mrmdd\P}{\mrmdd\Q} = \frac{1}{c \, |X|}\,.
\]
Then the sequence $\xi_j = X \eta_j$ ($j \geq 1$) is swap-invariant under~$\P$.
\end{proposition}

Note that in Proposition~\ref{swap-from-ex} we have $\P \sim \Q$. By Theorem~\cite[Theorem 17]{mss:14} there exists a random variable~$Y$ such that $\E_\Q |Y| < \infty$ and
\begin{equation}
\label{conv-eta-1}
n^{-1} \sum_{j = 1}^n \eta_j \to Y \quad\quad \mbox{a.s.}
\end{equation}
Then
\begin{equation}
\label{conv-xi-1}
n^{-1} \sum_{j = 1}^n \xi_j \to X Y \quad\quad \mbox{a.s.}
\end{equation}
If the convergence~(\ref{conv-eta-1}) is in~$L^1(\Q)$ (which is the case, for example, if $\eta$ is exchangeable and integrable under~$\Q$), then the convergence~(\ref{conv-xi-1}) is in~$L^1(\P)$.

\begin{example}[Division by first member of i.i.d.\ sequence]
\label{ex-swap-from-exch}
Let $\eta$ be a random sequence that is i.i.d.\ under~$\Q$ with $\E_\Q |\eta_1| < \infty$ and $\eta_1 \neq 0$ $\Q$-almost surely. Set $X = \eta_1^{-1}$ in Proposition~\ref{swap-from-ex}. Then the sequence $\xi_1 = 1$, $\xi_j = \eta_j / \eta_1 \; (j \geq 2)$ is swap-invariant under~$\P$. $\xi$~is not exchangeable under~$\P$. For the exchangeability of $\xi$ under $\P$ would imply that $\xi_2 \deq 1$ under~$\P$ and~$\Q$, whence $\eta_1 = \eta_2$ $\Q$-almost surely.
\end{example}

The following example shows that there exists a swap-invariant but not exchangeable sequence that takes only three values, in contrast to Theorem~\ref{lemma-swap-fin}.

\begin{example}
\label{ex-swap-from-exch-special}
Consider the special case of Example~\ref{ex-swap-from-exch} where $\eta_1$ takes values $1$ and $2$ with probability~$1/2$. We obtain the following finite-dimensional distributions of~$\xi$:
\begin{eqnarray*}
\lefteqn{ \P \!\left( \xi_2 = m_2, \ldots, \xi_n = m_n \right)}\\
 & = & \frac{2^{1 - n}}{3} \Big( \one \big\{ \, m_2, \ldots, m_n \in \left\{ 1, 2 \right\} \big\} + 2 \cdot \one \big\{ \, m_2, \ldots, m_n \in \left\{ 1/2, 1 \right\} \big\} \Big)
\end{eqnarray*}
where $n \geq 2$ and $m_2, \ldots, m_n \in \left\{ 1/2, 1, 2 \right\}$. An explicit calculation yields
\[
\E_\P \bigg| \sum_{j = 1}^n u_j \xi_j \bigg| \; = \; \frac{2^{1 - n}}{3} \sum_{ m_1, \ldots, \, m_n \in \left\{1, 2 \right\} } \bigg| \sum_{j = 1}^n u_j m_j \bigg|\,.
\]
So $\xi$ is swap-invariant under~$\P$.
\end{example}

\subsection{Ergodic representation}

As stated in~\cite[Theorem 17]{mss:14} for each swap-invariant sequence, the mean converges almost surely to an integrable random variable. We now demonstrate that, if the ergodic limit is different from zero and if the convergence is in $L^1$, the limit can be used to characterize swap-invariant sequences as scaled exchangeable sequences under another probability measure. 

\begin{theorem}
\label{ex-from-swap}
Let $\xi$ be a random sequence that is swap-invariant under a probability measure~$\P$ such that $n^{-1} \sum_{j = 1}^n \xi_j \to X\,$ $\P$-almost surely and in~$L^1(\P)$ as $n \to \infty$ with $\P(X \neq 0) = 1$. Then there exists a random sequence $\eta$ that is exchangeable and integrable under the probability measure $\Q$ defined by
\[
\frac{\mrmdd\Q}{\mrmdd\P} = \frac{ |X| }{\E_\P |X|}
\]
such that $\xi = X \eta\,$ $\P$-almost surely. 
\end{theorem}

\begin{proof}
Define
\[
\eta_j = \left\{ \begin{array}{ll} \xi_j / X & \quad \mathrm{on} \;\left\{X \neq 0 \right\} \\[1em]
 0 & \quad \mathrm{on} \;\left\{X = 0 \right\} \end{array}, \quad j \geq 1.
\right.
\]
For $j \geq 1$, $\E_{\P} | \xi_j | < \infty$ implies $\E_{\Q} | \eta_j | < \infty$. Let $n \geq 1$, $u \in \R^n$, and $\pi \in \perm(n)$. For $m \geq n$, the swap-invariance of~$\xi$ under~$\P$ yields
\begin{eqnarray*}
\E_{\P} \bigg| m^{-1} \sum_{k = 1}^m \xi_k + \sum_{j = 1}^n u_j \xi_j \bigg| & = & \E_{\P}  \bigg| m^{-1} \sum_{k = n+1}^m \xi_k + \sum_{j = 1}^n \left(u_j + m^{-1}\right) \xi_j \bigg| \\
& = & \E_{\P}  \bigg| m^{-1} \sum_{k = n+1}^m \xi_k + \sum_{j = 1}^n \left(u_{\pi(j)} + m^{-1}\right) \xi_j \bigg| \\
& = & \E_{\P} \bigg| m^{-1} \sum_{k = 1}^m \xi_k + \sum_{j = 1}^n u_{\pi(j)} \xi_j \bigg|\,.
\end{eqnarray*}
Letting $m \to \infty$ we obtain
\begin{equation*}
\E_{\P} \bigg| X + \sum_{j = 1}^n u_j \xi_j \bigg| = \E_{\P} \bigg| X + \sum_{j = 1}^n u_{\pi(j)} \xi_j \bigg|\,.
\end{equation*}
After change of measure this gives $\E_\Q \big| 1 + \sum_{j = 1}^n u_j \eta_j \big| = \E_\Q \big| 1 + \sum_{j = 1}^n u_{\pi(j)} \eta_j \big|$. Since this holds for all $u \in \R^n$, it follows by~\cite[Theorem 1.1]{har:81} that $\eta$ is exchangeable under~$\Q$.
\end{proof}

\begin{example}
\label{ex-swap-lognormal}
Let $(Z_j)_{j \geq 1}$ be i.i.d.\ standard normal random variables under a probability measure~$\P$, and let $(b_j)_{j \geq 1}$ be a sequence of real numbers such that $\beta = \sum_{j \geq 1} b_j^2 < \infty$. Define a random sequence~$(\xi)_{j \geq 1}$ by $\xi_j = \exp \zeta_j$ where
\[
\zeta_j = Z_j + \sum_{k = 1}^\infty b_k Z_k + \mu_j \, , \quad \quad \mu_j = - \frac{1}{2} \left( 1 + \beta  + 2 b_j \right).
\]
Note that $\xi_i \deq \xi_j$ if and only if $b_i = b_j$\,. In~\protect{\cite[Examples 15 and 25]{mss:14}} it is shown that $\xi$ is swap-invariant under~$\P$ and that the limit
\[
X = \lim_{n \to \infty} n^{-1} \sum_{j = 1}^n \xi_j = \exp \left( \sum_{k = 1}^\infty b_k Z_k - \frac{1}{2} \beta \right)
\]
exists $\P$-almost surely and in~$L^1(\P)$. By Theorem~\ref{ex-from-swap} the random sequence
\[
\eta_j = \frac{\xi_j}{X} = \exp \left(Z_j - b_j - \frac{1}{2} \right)
\]
is exchangeable under~$\Q$.
\end{example}

In the context of Theorem~\ref{ex-from-swap} Birkhoff's ergodic theorem implies a relation between $\xi_1$ and $X$ that we state in the following corollary. As usual the $\sigma$-algebra generated by the open sets in~$\R^\infty$ is denoted by~$\sB(\R^\infty)$, the tail $\sigma$-algebra of~$\sB(\R^\infty)$ by~$\sT$, the shift-invariant $\sigma$-algebra by~$\sI$, and the $\sigma$-algebra of sets that are invariant under all finite permutations by~$\sE$. For a random sequence~$\eta$ we define the corresponding $\sigma$-algebras on the basic probability space by
\[
\sT_\eta = \eta^{-1} \sT, \quad\quad \sI_\eta = \eta^{-1} \sI, \quad\quad \sE_\eta = \eta^{-1} \sE\,.
\]
It is well known that, if $\eta$ is exchangeable, then $\sT_\eta = \sI_\eta = \sE_\eta$\, almost surely, i.e.\ their completions are equal (see for example \protect{\cite[Corollary 1.6]{kallenberg:sym}}). The sign of~$x \in \R$ is denoted by $\mathrm{sign}(x)$.

\begin{corollary}
\label{cor-si-signx}
Under the conditions of Theorem~\ref{ex-from-swap}
\[
\E_\P \! \left[ \mathrm{sign} (X) \xi_j \middle\vert \sT_\eta \right] = \E_\P \! \left[ \, |X| \, \middle\vert \sT_\eta \right] \quad\quad \mbox{a.s.}\,,\, j \geq 1,
\]
where $\eta$ is defined as in Theorem~\ref{ex-from-swap}.
\end{corollary}

\begin{proof}
First note that $\P$ and $\Q$ are equivalent, so we may say that an equality or convergence holds `almost surely' without specifying the probability measure. Now on the one hand the definition of~$\eta$ implies that $n^{-1} \sum_{k = 1}^n \eta_k \to 1$ almost surely as $n \to \infty$. On the other hand \protect{\cite[Theorem 10.6]{kallenberg:fou}} yields
\[
n^{-1} \sum_{k = 1}^n \eta_k \to \E_\Q \! \left[ \eta_j \middle\vert \sT_\eta \right] = \frac{\E_\P \! \left[ \, |X| \eta_j \middle\vert \sT_\eta \right]}{\E_\P \!\left[ \, |X| \, \middle\vert \sT_\eta \right]} = \frac{\E_\P \! \left[ \mathrm{sign}(X) \xi_j \middle\vert \sT_\eta \right]}{\E_\P \! \left[ \, |X| \, \middle\vert \sT_\eta \right]} \quad\quad \mbox{a.s.}\,,\, j \geq 1. \qedhere
\]
\end{proof}

\subsection{$p$-norm representation}
\label{subsec-norms}

The connection between swap-invariant sequences and exchangeable sequences established in the preceding section is restricted to cases where the ergodic limit is attained in~$L^1$ and is almost surely different from zero. Now we present another method, using $p$-norms, where the second condition can be dropped. However $L^1$-convergence is still required.

For a sequence~$x \in \R^\infty$ the vector of the first $n$ components of~$x$ is denoted by $x^{(n)} = (x_1, \ldots, x_n)$. We define, for $x \in \R^\infty$, $n \geq 1 $, and $1 \leq p < \infty$,
\[
\left\| x \right\|_p^{(n)} = \bigg( n^{-1}\sum_{j=1}^n \left| x_j \right|^p \bigg)^{1/p}, \quad \quad \left\| x \right\|_p = \limsup_{n \to \infty} \left\| x \right\|_p^{(n)} \in \overline{\R}_+ \,,
\]
and
\[
\left\| x \right\|_\infty^{(n)} = \max\left\{ \left| x_j \right| ; 1 \leq j \leq n \right\}, \quad \quad \left\| x \right\|_\infty = \limsup_{n \to \infty} \left\| x \right\|_\infty^{(n)} \in \overline{\R}_+\,.
\]

For the proof of the main result in this section, Theorem~\ref{exch-from-sym}, the lemma below is required. For its proof we need the following convergence result.

\begin{proposition}
\label{prop-l1-conv-zero}
Let $X, Y, X_n\,, Y_n$ ($n \geq 1$) be non-negative random variables such that $X_n \to X$ as $n \to \infty$ almost surely and in~$L^1$, $Y_n$ and $Y$ are bounded by $K$ for some $K > 0$, and $Y_n \to Y$ almost surely on the event~$\left\{ X > 0 \right\}$. Then $X_n Y_n \to X Y$ almost surely and in~$L^1$.
\end{proposition}

\begin{proof}
We have
\begin{eqnarray*}
\E \left| X_n Y_n - X Y \right| & = & \E \left| (X_n - X) (Y_n -Y) \right| \, + \, \E \left| X (Y_n - Y) \right| \, + \, \E \left| (X_n - X) Y \right| \\[.5em]
 & \leq & 2K \, \E \left| X_n - X \right| \, + \, \E \big[ \left| X (Y_n - Y) \right| \,\one \!\left\{ X > 0 \right\} \big] \, + \, K \, \E \left| X_n - X \right|\,.
\end{eqnarray*}
All three terms on the right-hand side converge to zero as $n \to \infty$.
\end{proof}

\begin{lemma}
\label{prop-norm}
Fix $p \in [1,\infty]$. Let $\xi_1\,$, $\xi_2$ be two symmetric sequences of random variables such that $\xi_1^{(n)}$ and $\xi_2^{(n)}$ are zonoid equivalent under~$\P$ for each $n \geq 1$, and
\[
\left\| \xi_i \right\|_p^{(n)} \to \left\| \xi_i \right\|_p \quad\quad \mbox{$\P$-a.s.\ and in~$L^1(\P)$ as $n \to \infty$}\,.
\]
Then $\E_\P \!\left\| \xi_1 \right\|_p = \E_\P \!\left\| \xi_2 \right\|_p$\,, and if $\E_\P \!\left\| \xi_1 \right\|_p > 0$, the probability measures $\Q_i$ defined by
\[
\frac{\mrmdd\Q_i}{\mrmdd\P} = \frac{\left\| \xi_i \right\|_p}{\E_\P \!\left\| \xi_i \right\|_p},\quad i = 1, 2,
\]
satisfy $\Q_1 \big( \xi_1 / \!\left\| \xi_1 \right\|_p \in B \big) = \Q_2 \big( \xi_2 / \!\left\| \xi_2 \right\|_p \in B \big)$ for all $B \in \sB(\R^\infty)$.
\end{lemma}

\begin{proof}
To simplify notation we omit the subscript~$p$ at $\left\| x \right\|_p^{(n)}$ and $\left\| x \right\|_p$ for $x \in \R^\infty$ and $n \geq 1$ throughout the proof.

Let $n \geq 1$. Since $\xi_1^{(n)}$ and $\xi_2^{(n)}$ are zonoid equivalent, Lemma~\ref{lemma-norm-ze} implies that $\E_\P \!\left\| \xi_1 \right\|^{(n)} = \E_\P \!\left\| \xi_2 \right\|^{(n)}$. Letting $n \to \infty$ yields $\E_\P \!\left\| \xi_1 \right\| = \E_\P \!\left\| \xi_2 \right\|$\,.

Now assume $\E_\P \!\left\| \xi_1 \right\| > 0$. Choose $N \geq 1$ such that $\E_\P \!\left\| \xi_1 \right\|^{(n)} > 0$ for $n \geq N$. For $n \geq N$ and $i \in \left\{ 1, 2 \right\}$ define probability measures~$\Q_i^{(n)}$ by
\[
\frac{\mrmdd\Q_i^{(n)}}{\mrmdd\P} = \frac{\left\| \xi_i \right\|^{(n)} }{\E_\P \!\left\| \xi_i \right\|^{(n)}}\,.
\]
By Lemma~\ref{lemma-norm-ze}
\[
\Q_1^{(n)} \big( \xi_1^{(n)} / \left\| \xi_1 \right\|^{(n)} \in A \,\big) = \,\Q_2^{(n)} \big( \xi_2^{(n)} / \left\| \xi_2 \right\|^{(n)} \in A \,\big)
\]
for $A \in \sB(\R^n)$. It follows that, for $n \geq N$, $1 \leq k \leq n$, and $A \in \sB(\R^k)$,
\begin{eqnarray}
\label{eq-distr}
\Q_1^{(n)} \!\left( \frac{\xi_1^{(k)}}{\left\| \xi_1 \right\|^{(n)}} \in A \right) \!\! & = & \!
\Q_1^{(n)} \!\left( \frac{\xi_1^{(n)}}{\left\| \xi_1 \right\|^{(n)}} \in A \times \R^{n - k} \right) \\[.5em]
\! & = & \! \Q_2^{(n)} \!\left( \frac{\xi_2^{(n)}}{\left\| \xi_2 \right\|^{(n)}} \in A \times \R^{n - k} \right) = \,
\Q_2^{(n)} \!\left( \frac{\xi_2^{(k)}}{\left\| \xi_2 \right\|^{(n)}} \in A \right). \nonumber
\end{eqnarray}
Now let $f$ be a bounded continuous function from~$\R^k$ to~$\R_+$\,. Then, for $i \in \left\{ 1, 2 \right\}$,
\begin{equation}
\label{equ-bound-xi}
\E_i^{(n)} f \left( \frac{\xi_i^{(k)}}{\left\| \xi_i \right\|^{(n)}} \right) = \frac{1}{\E_\P \!\left\| \xi_i \right\|^{(n)}} \,
 \int\limits_{ \left\{ \left\| \xi_i \right\|^{(n)} > 0 \right\} } \left\| \xi_i \right\|^{(n)} f \left( \frac{\xi_i^{(k)}}{\left\| \xi_i \right\|^{(n)}} \right) d\P
\end{equation}
where $\E_i^{(n)}$ denotes the expectation with respect to~$\Q_i^{(n)}$. Now we apply Proposition~\ref{prop-l1-conv-zero} to the random variables
\begin{eqnarray*}
X_n & \!\! = \!\! & \left\| \xi_i \right\|^{(n)}, \quad\quad X \, = \, \left\| \xi_i \right\|, \\[1em]
Y_n & \!\! = \!\! & \left\{ \begin{array}{ll} f \left( \displaystyle \frac{\xi_i^{(k)}}{\left\| \xi_i \right\|^{(n)}} \right) & \quad \mathrm{on} \;\, \big\{ \!\left\| \xi_i \right\|^{(n)} > 0 \big\} \\[1em]
 0 & \quad \mathrm{on} \;\, \big\{ \!\left\| \xi_i \right\|^{(n)} = 0 \big\} \end{array} \right. ,\\[1em]
Y & \!\! = \!\! & \left\{ \begin{array}{ll} f \left( \displaystyle \frac{\xi_i^{(k)}}{\left\| \xi_i \right\|} \right) & \quad \mathrm{on} \;\left\{ \left\| \xi_i \right\| > 0 \right\} \\[1em]
 0 & \quad \mathrm{on} \;\left\{\left\| \xi_i \right\| = 0 \right\} \end{array} \right. .
\end{eqnarray*}
In particular, note that for a point $\omega$ with $\left\| \xi_i (\omega) \right\| > 0$ we have $\left\| \xi_i (\omega) \right\|^{(n)} > 0$ for all sufficiently large~$n$. Hence the continuity of~$f$ implies that $Y_n \to Y$ as $n \to \infty$ $\P$-almost surely on the event $\left\{ \left\| \xi_i \right\| > 0 \right\}$. We conclude that the right-hand side of~\eqref{equ-bound-xi} converges to
\[
\frac{1}{\E_\P \!\left\| \xi_i \right\|} \int\limits_{ \left\{ \left\| \xi_i \right\| > 0 \right\} } \left\| \xi_i \right\| f \left( \frac{\xi_i^{(k)}}{\left\| \xi_i \right\|} \right) d\P \,=\, \E_i \, f \left( \frac{\xi_i^{(k)}}{\left\| \xi_i \right\|} \right)
\]
where $\E_i$ denotes the expectation with respect to~$\Q_i$\,. Using equality of the distributions in~\eqref{eq-distr} yields $\E_1 \, f \big( \xi_1^{(k)} / \left\| \xi_1 \right\| \big) = \E_2 \, f  \big( \xi_2^{(k)} / \left\| \xi_2 \right\| \big)$. By approximation it follows that $\Q_1 \big( \xi_1^{(k)} / \left\| \xi_1 \right\| \in A \big) = \,\Q_2 \big( \xi_2^{(k)} / \left\| \xi_2 \right\| \in A \big)$ for each $A \in \sB(\R^k)$. Since this holds for all $k \geq 1$, we obtain $\Q_1 \big( \xi_1 / \left\| \xi_1 \right\| \in B \big) = \,\Q_2 \big( \xi_2 / \left\| \xi_2 \right\| \in B \big)$ for each $B \in \sB(\R^\infty)$.
\end{proof}

\begin{theorem}
\label{exch-from-sym}
Fix $p \in [1,\infty]$. Let $\xi$ be a random sequence that is swap-invariant under a probability measure~$\P$ such that $\left\| \xi \right\|_p^{(n)} \to \left\| \xi \right\|_p$ as $n \to \infty$ $\P$-almost surely and in~$L^1(\P)$, and \mbox{$\E_\P \!\left\| \xi \right\|_p > 0$}. Define another probability measure~$\Q$ by
\[
\frac{\mrmdd\Q}{\mrmdd\P} = \frac{\left\| \xi \right\|_p}{\E_\P \!\left\| \xi \right\|_p}\,.
\]
\begin{benumber}
\item \label{exch-from-sym-general} There exists a random sequence $\zeta$ that is exchangeable under~$\Q$ such that
\[
\varepsilon \, \xi_j = \left\| \xi \right\|_p \zeta_j \quad\quad \mbox{$\P$-a.s.\,, $j \geq 1$}
\]
where $\varepsilon$ is a random variable that takes values $\pm 1$ with probabilities $1/2$ and is independent of~$\xi$ under~$\P$.
\item \label{exch-from-sym-possym} If $\xi$ is either symmetric or non-negative, then there exists a random sequence $\eta$ that is exchangeable under~$\Q$ such that
\[
\xi_j = \left\| \xi \right\|_p \eta_j \quad\quad \mbox{$\P$-a.s.\,, $j \geq 1$}\,.
\]
\end{benumber}
\end{theorem}

\begin{proof}
To simplify notation we again omit the subscript~$p$ at $\left\| x \right\|_p^{(n)}$ and $\left\| x \right\|_p$ for $x \in \R^\infty$ and $n \geq 1$.

First assume that $\xi$ is symmetric under~$\P$. Let $d \geq 1$ and $\pi \in \perm(d)$, and denote by $\xi \circ \pi$ the random sequence that results from $\xi$ by applying the permutation~$\pi$ to the first $d$~members. Then $\left\| \xi \right\| = \left\| \xi \circ \pi \right\|$. Moreover $\xi^{(n)}$ and $(\xi \circ \pi)^{(n)}$ are zonoid equivalent under~$\P$ for all $n \geq 1$. Define
\[
\eta_j = \left\{ \begin{array}{ll} \displaystyle \frac{\xi_j}{\left\| \xi \right\|} & \quad \mathrm{on} \;\left\{\left\| \xi \right\| > 0 \right\} \\[1em]
 0 & \quad \mathrm{on} \;\left\{\left\| \xi \right\| = 0 \right\} \end{array}
\right. , \quad j \geq 1.
\]
By Lemma~\ref{prop-norm}
\[
\Q \big( \eta \in B \big) = \Q \!\left( \frac{\xi}{\left\| \xi \right\|} \in B \right) = \,\Q \!\left( \frac{\xi \circ \pi}{\left\| \xi \right\|} \in B \right) = \Q \big( \eta \circ \pi \in B \big)\,.
\]
for all $B \in \sB(\R^\infty)$. This shows that $\eta$ is exchangeable under~$\Q$. It remains to show that $\xi_j = 0$ $\P$-almost surely on the event $\left\{\left\| \xi \right\| = 0 \right\}$ for all $j \geq 1$. Define random sequences $\xi^>$ and $\xi^=$ by
\[
\xi^>_j = \xi_j \, \one \!\left\{\left\| \xi \right\| > 0 \right\}, \quad\quad \xi^=_j = \xi_j \, \one \!\left\{\left\| \xi \right\| = 0 \right\}
\]
for $j \geq 1$. For $n \geq 1$ and $u \in \R^n$ we have
\begin{eqnarray*}
\E_\P \bigg|\sum_{j = 1}^n u_j \xi^>_j \bigg| & \! = \! & \E_\P \one \!\left\{\left\| \xi \right\| > 0 \right\} \bigg|\sum_{j = 1}^n u_j \xi_j \bigg| \\[.5em]
 & \! = \! & \E_\P \!\left\| \xi \right\| \bigg| \sum_{j = 1}^n u_j \eta_j \bigg| \, = \, \E_\P \!\left\| \xi \right\| \E_\Q \bigg|\sum_{j = 1}^n u_j \eta_j \bigg|.
\end{eqnarray*}
Hence $\xi^>$ is swap-invariant under~$\P$. This implies that also $\xi^=$ is swap-invariant under~$\P$. Since $\left\| \xi \right\|^{(n)} \to \left\| \xi \right\|\,$ $\P$-almost surely and in~$L^1(\P)$,
\[
\left\| \xi^= \right\|^{(n)} \to \left\| \xi \right\| \, \one \!\left\{\left\| \xi \right\| = 0 \right\} = 0 \quad\quad \mbox{$\P$-a.s.\ and in~$L^1(\P)$}\,.
\]
Since $\left\| \xi^= \right\|_1^{(n)} \leq \left\| \xi^= \right\|^{(n)}$ by the H\"{o}lder inequality, we conclude that $\left\| \xi^= \right\|_1^{(n)} \to 0\,$ $\P$-almost surely and in~$L^1(\P)$. The swap-invariance of~$\xi^=$ implies that $\E_\P \!\left\| \xi^= \right\|_1^{(n)} = \E_\P |\xi_1^=|$ for all $n \geq 1$. Therefore $\E_\P |\xi_j^=| = \E_\P |\xi_1^=| = 0$, hence $\xi_j^= = 0$ $\P$-almost surely.

If $\xi$ is not symmetric, we may define a random sequence $\rho$ by $\rho_j = \varepsilon \, \xi_j$ for $j \geq 1$ where $\varepsilon$ has the stated properties. Then $\rho$ is symmetric and swap-invariant under~$\P$. Applying the preceding proof to~$\rho$ proves~\eqref{exch-from-sym-general}. In particular, if $\xi$ is non-negative, then also the sequence $(|\zeta_1|, |\zeta_2|, \ldots)$ is exchangeable under~$\Q$, which shows the second statement of~\eqref{exch-from-sym-possym}.
\end{proof}

In the following example a symmetric random sequence $\xi$ is defined that is swap-invariant but not exchangeable. Since each component as well as the ergodic limit is zero with positive probability, neither \protect{\cite[Theorem 21]{mss:14}} nor Theorem~\ref{ex-from-swap} can be used to obtain a representation in terms of an exchangeable sequence. However Theorem~\ref{exch-from-sym} can be applied.

\begin{example}
\label{ex-swap-norm}
Let $\rho$ be a random sequence that is i.i.d.\ under a probability measure~$\pR$ such that $\rho_1$ takes values $-1$, $0$, $+1$ with equal probability~$1/3$. Further let $X = 1 + |\rho_1|$. Define another probability measure $\P$ by
\[
\frac{\mrmdd\P}{\mrmdd\pR} = \frac{3}{2} X^{-1},
\]
and a random sequence $\xi_j = X \rho_j$ ($j \geq 1$). Then $\xi$ is swap-invariant under~$\P$ by Proposition~\ref{swap-from-ex} and has finite-dimensional distributions
\begin{eqnarray*}
\lefteqn{ \P \!\left( \xi_1 = m_1, \ldots, \xi_n = m_n \right)}\\[.5em]
 & =\! & 3^{1 - n} \cdot \bigg( \, \frac{1}{2} \, \one \big\{ m_1 = 0,\; m_2, \ldots, m_n \in \left\{ -1, 0, 1 \right\} \big\} \\[.5em]
 &&  \quad \quad \quad \quad + \;\frac{1}{4} \, \one \big\{ \, m_1 \in \left\{ -2, 2 \right\}, \; m_2, \ldots, m_n \in \left\{ -2, 0, 2 \right\} \big\} \bigg)
\end{eqnarray*}
where $n \geq 1$ and $m_1, \ldots, m_n \in \left\{ -2, -1, 0, 1, 2 \right\}$. In particular,
\[
\P \!\left( \xi_1 = 0 \right) \, = \, \frac{1}{2} \, , \quad \P \!\left( \xi_1 = \pm 2 \right) \, = \, \frac{1}{4} \, ,
\]
and, for $j \geq 2$,
\[
\P \!\left( \xi_j = 0 \right) \, = \, \frac{1}{3} \, , \quad \P \!\left( \xi_j = \pm 1 \right) \, = \, \P \!\left( \xi_j = \pm 2 \right) \, = \, \frac{1}{6} \, .
\]
Thus $\xi$ is symmetric and not exchangeable under~$\P$. Since $\P \!\left( \xi_j = 0 \right) > 0$ for all $j$, we cannot divide the sequence by one of its members in order to obtain an exchangeable sequence as done in~\protect{\cite[Theorem 21]{mss:14}}. Further note that
\[
n^{-1} \sum_{j = 1}^n \rho_j \; \to \; \E_\pR \rho_1 = 0 \quad \quad \mbox{$\pR$-a.s.}\,,
\]
which implies $n^{-1} \sum_{j = 1}^n \xi_j \; \to \; 0\,$ $\P$-almost surely. Thus Theorem~\ref{ex-from-swap} is not applicable here either. In order to apply
Theorem~\ref{exch-from-sym} fix $p = 1$. Since the sequence $(\left| \rho_j \right|)_{j \geq 1}$ is i.i.d.\ under~$\pR$,
\[
n^{-1} \sum_{j = 1}^n |\rho_j| \; \to \; \E_\pR |\rho_1| = \frac{2}{3} \quad \quad \mbox{$\pR$-a.s.\ and in~$L^1(\pR)$}\,.
\]
It follows that
\[
\left\| \xi \right\|_1^{(n)} = X n^{-1} \sum_{j = 1}^n \left| \rho_j \right| \; \to \;  \frac{2}{3} X = \left\| \xi \right\|_1 \quad \quad \mbox{$\P$-a.s.\ and in~$L^1(\P)$}
\]
and $\E_\P  \!\left\| \xi \right\|_1 = 1$. Thus the conditions of Theorem~\ref{exch-from-sym} are satisfied. We apply the definitions in Theorem~\ref{exch-from-sym},
\[
\frac{\mrmdd\Q}{\mrmdd\P} = \frac{2}{3} X, \quad \quad \eta_j = \frac{3}{2} \rho_j\,,\quad j \geq 1,
\]
and find that $\Q = \pR$. Theorem~\ref{exch-from-sym} says that $\eta$ is exchangeable under~$\Q$, which can be immediately confirmed here.
\end{example}

An interesting special case of Theorem~\ref{exch-from-sym} is that of non-negative sequences and~\mbox{$p = 1$}. In this case the limit in Theorem~\ref{exch-from-sym} is the ordinary ergodic limit and the probability measure~$\Q$ is defined as in Theorem~\ref{ex-from-swap}. However the conditions are weaker than in Theorem~\ref{ex-from-swap} because the ergodic limit can be zero with positive $\P$-probability here.

\begin{example}
The lognormal sequence in Example~\ref{ex-swap-lognormal} satisfies the assumptions of Theorem~\ref{ex-from-swap}, and therefore also those of Theorem~\ref{exch-from-sym} for $p = 1$.
\end{example}

A slightly more general case is $p = 1$ and no sign restrictions on~$\xi$. Application of Theorem~\ref{exch-from-sym} requires the mean of absolute values $n^{-1} \sum_{j = 1}^n |\xi_j|$ to converge almost surely and in~$L^1$. However if $\xi$ is swap-invariant, it follows from Lemma~\ref{zon-equ-abs} that also the sequence of absolute values, $(|\xi_1|, |\xi_2|, \ldots)$, is swap-invariant, so the almost sure convergence is guaranteed by~\cite[Theorem 17]{mss:14}. Therefore only the $L^1$-convergence remains to be checked. By Scheff\'{e}'s lemma, this reduces to the condition $\E \left| \xi_1 \right| = \E \left\| \xi \right\|_1$\,. Also note that the inequality $\E \left| \xi_1 \right| \geq \E \left\| \xi \right\|_1$ is always guaranteed by Fatou's lemma. We summarize the situation for the case $p = 1$ as follows:

\begin{proposition}
\label{prop-si-p1}
Let $\xi$ be a swap-invariant random sequence. Then $n^{-1} \sum_{j = 1}^n |\xi_j|$ converges almost surely to an integrable random variable $\left\| \xi \right\|_1$ as $n \to \infty$. If $\E \left| \xi_1 \right| = \E \left\| \xi \right\|_1$\,, then this convergence is in~$L^1$.
\end{proposition}

From Theorem~\ref{exch-from-sym} we finally derive a representation of the ergodic limit of symmetric or non-negative swap-invariant sequences. Again the general case is obtained by noting that, for a swap-invariant sequence~$\xi$, the symmetric sequence $\varepsilon \, \xi$ is swap-invariant as well. In the derivation of Theorem~\ref{erg-limit-norm} the formula for the conditional expectation under a change of the probability measure is used in the following form where the Radon-Nikod\'{y}m derivative may be zero with non-zero probability.

\begin{proposition}
\label{prop-change-cond}
Let $(\Omega, \sF, \P)$ be a probability space, $\sG$ a sub-$\sigma$-algebra of~$\sF$, $Z$~ a random variable with $Z \geq 0$ and $\E_\P Z = 1$, $\Q$ another probability measure defined by ${\mrmdd\Q} / {\mrmdd\P} = Z$, and $Y$ a random variable with $\E_\Q |Y| < \infty$. Then $\E_\P |ZY| < \infty$, and
\[
\one \!\left\{ Z_0 > 0 \right\} \E_\Q \!\left[ Y \middle\vert \sG \right] = \hat{Z}_0 \, \E_\P \!\left[ ZY \middle\vert \sG \right] \quad\quad \mbox{$\Q$-a.s.}
\]
where
\[
Z_0 = \E_\P \!\left[ Z \middle\vert \sG \right], \quad \hat{Z}_0 = \left\{ \begin{array}{ll} Z_0^{-1} & \quad \mathrm{on} \;\left\{Z_0 > 0 \right\} \\[1em]
 0 & \quad \mathrm{on} \;\left\{Z_0 = 0 \right\} \end{array}
\right. \,.
\]
\end{proposition}

\begin{theorem}
\label{erg-limit-norm}
Let $\xi$ be a symmetric or non-negative sequence of random variables that satisfies the conditions of Theorem~\ref{exch-from-sym} for some $p \in \left[1, \infty\right]$, and choose $\eta$ as in Theorem~\ref{exch-from-sym}~\eqref{exch-from-sym-possym}. Then
\begin{equation}
\label{swap-ergodic-equ}
n^{-1} \sum_{j = 1}^n \xi_j \;\to\; \left\| \xi \right\|_p \hat{Y}_0 \; \E_\P \!\left[ \xi_1 \middle\vert \sT_\eta \right] \quad\quad  \mbox{$\P$-a.s.\ and in~$L^1(\P)$ as $n \to \infty$},
\end{equation}
where
\[
Y_0 = \E_\P \big[ \left\| \xi \right\|_p \big\vert \sT_\eta \big], \quad \hat{Y}_0 = \left\{ \begin{array}{ll} Y_0^{-1} & \quad \mathrm{on} \;\left\{Y_0 > 0 \right\} \\[1em]
 0 & \quad \mathrm{on} \;\left\{Y_0 = 0 \right\} \end{array}
\right. \,.
\]
Moreover,
\[
\big\{ \!\left\| \xi \right\|_p > 0 \big\} \in \sT_\eta\,, \quad\quad \P \Big( \!\left\{ Y_0 > 0 \right\} \Delta \, \big\{ \!\left\| \xi \right\|_p > 0 \big\} \Big) = 0 \,.
\]
\end{theorem}

Theorem~\ref{erg-limit-norm} says that under the stated conditions
\[
n^{-1} \sum_{j = 1}^n \xi_j \;\to \; \frac{ \left\| \xi \right\|_p \E_\P \! \left[ \xi_1 \middle\vert \sT_\eta \right] }{ \E_\P \big[ \left\| \xi \right\|_p \big\vert \sT_\eta \big] } \quad\quad \mbox{$\P$-a.s.~on} \; \big\{ \!\left\| \xi \right\|_p > 0 \big\} \; \mbox{as $n \to \infty$}.
\]
In contrast to the representations of the ergodic limit in Theorem~21 and Proposition~22 in~\cite{mss:14}, we may allow $\P \!\left( \xi_j = 0 \right) > 0$ for all $j$ here. For non-negative $\xi$ and $p = 1$, we summarize the results of Theorems~\ref{exch-from-sym} and~\ref{erg-limit-norm} in the following corollary; note that this includes variants of Theorem~\ref{ex-from-swap} and Corollary~\ref{cor-si-signx}.

\begin{corollary}
\label{cor-nonneg-seq}
Let $\xi$ be a sequence of non-negative random variables that is swap-invariant under a probability measure~$\P$ such that
\[
n^{-1} \sum_{j = 1}^n \xi_j \to X \quad\quad \mbox{$\P$-a.s.\ and in~$L^1(\P)$ as $n \to \infty$}
\]
with $\E_\P X > 0$. Define the probability measure $\Q$ by
\[
\frac{\mrmdd\Q}{\mrmdd\P} = \frac{X}{\E_\P X}\,.
\]
Then there exists a random sequence $\eta$ that is exchangeable under~$\Q$ such that $\xi_j = X \eta_j$ $\P$-almost surely for $j \geq 1$. Moreover $\E_\P \!\left[ X \middle\vert \sT_\eta \right] = \E_\P \!\left[ \xi_1 \middle\vert \sT_\eta \right]$ $\P$-almost surely.
\end{corollary}

\begin{proof}[Proof of Theorem~\ref{erg-limit-norm}]
Since $\eta$ is $\Q$-integrable and exchangeable under~$\Q$, it follows by~\protect{\cite[Theorem 10.6]{kallenberg:fou}} that
\[
n^{-1} \sum_{j = 1}^n \eta_j \to \E_\Q \!\left[ \eta_1 \middle\vert \sT_\eta \right] \quad\quad \mbox{$\Q$-a.s.\ and in~$L^1(\Q)$}\,.
\]
By Proposition~\ref{prop-change-cond},
\begin{equation}
\label{erg-from-birk}
\one \!\left\{ Y_0 > 0 \right\} \, \E_\Q \!\left[ \eta_1 \middle\vert \sT_\eta \right] \; = \; \hat{Y}_0 \; \E_\P \!\left[ \xi_1 \middle\vert \sT_\eta \right] \quad\quad \mbox{$\Q$-a.s.}
\end{equation}
We write $\left\| \xi \right\|$ for $\left\| \xi \right\|_p$ in the following. Multiplying both sides of~\eqref{erg-from-birk} by $\left\| \xi \right\|$ shows that
\begin{equation}
\label{erg-with-y}
\one \!\left\{ Y_0 > 0 \right\} \, n^{-1} \sum_{j = 1}^n \xi_j \;\; \to \;\; \left\| \xi \right\| \hat{Y}_0 \; \E_\P \!\left[ \xi_1 \middle\vert \sT_\eta \right] \quad\quad \mbox{$\Q$-a.s.~as $n \to \infty$}\,.
\end{equation}
In order to see that the convergence \eqref{erg-with-y} holds $\P$-almost surely, define $E = \left\{ \left\| \xi \right\| > 0 \right\}$. On~$E^c$ we have $\xi_j = 0$ $\P$-almost surely for all $j \geq 1$. On~$E$ the measures $\P$ and $\Q$ are equivalent, so the convergence holds also $\P$-almost surely on~$E$.

We now show that $E \in \sT_\eta$\,. We may assume that $\eta_j = 0$ on~$E^c$ for $j \geq 1$. Define $A_n = \left\{ \eta_j = 0 \, ; \, j \geq n \right\}$ for $n \geq 1$. Note that
\[
A_n \subset \left\{ \xi_j = 0 \, ; \, j \geq n \right\} \subset E^c,
\]
and $E^c \subset A_n$ by assumption. Hence $E^c = A_n$ for all $n \geq 1$, and therefore $E^c \in \sT_\eta$\,.

We next show the last statement, which then implies that the convergence~\eqref{swap-ergodic-equ} holds $\P$-almost surely. Now $\E_\P \!\left[ Y_0 \one_A \right] = \E_\P \!\left[ \, \left\| \xi \right\| \!\one_A \right]$ for each $A \in \sT_\eta$ by definition of conditional expectation. Define $F = \left\{ Y_0 > 0 \right\}$. Note that \mbox{$\E_\P \!\left[ Y_0 \one_A \right] > 0$} if and only if \mbox{$\P \!\left( F \cap A \right) > 0$}, and \mbox{$\E_\P \!\left[ \, \left\| \xi \right\| \!\one_A \right] > 0$} if and only if \mbox{$\P \!\left( E \cap A \right) > 0$}. We conclude that \mbox{$\P \!\left( F \cap A \right) > 0$} if and only if \mbox{$\P \!\left( E \cap A \right) > 0$}. It follows that $\P \!\left( F^c \setminus E^c \right) = \P \!\left( F^c \cap E \right) = 0$ because $F^c \in \sT_\eta$\,. Moreover $\P \!\left( E^c \setminus F^c \right) = \P \!\left( E^c \cap F \right) = 0$ because $E^c \in \sT_\eta$\,. Thus $\P \!\left( F^c \Delta E^c \right) = 0$.

To see that the convergence~\eqref{swap-ergodic-equ} holds in~$L^1(\P)$ note that $\one \!\left\{ Y_0 > 0 \right\} = \one \!\left\{ \left\| \xi \right\| > 0 \right\}$ $\P$-almost surely and therefore also $\Q$-almost surely. It follows that
\begin{eqnarray*}
\E_\P \bigg\vert n^{-1} \sum_{j = 1}^n \xi_j \; - \; 
 \left\| \xi \right\| \hat{Y}_0 \; \E_\P \!\left[ \xi_1 \middle\vert \sT_\eta \right] \bigg\vert
& = & \E_\P \!\left\| \xi \right\| \bigg\vert n^{-1} \sum_{j = 1}^n \eta_j - \hat{Y}_0 \; \E_\P \!\left[ \xi_1 \middle\vert \sT_\eta \right] \bigg\vert \\
 & = & \frac{1}{\E_\P \!\left\| \xi \right\|} \, \E_\Q \bigg\vert n^{-1} \sum_{j = 1}^n \eta_j - \hat{Y}_0 \; \E_\P \!\left[ \xi_1 \middle\vert \sT_\eta \right] \bigg\vert \\
 & = & \frac{1}{\E_\P \!\left\| \xi \right\|} \, \E_\Q \bigg\vert n^{-1} \sum_{j = 1}^n \eta_j - \one \!\left\{ Y_0 > 0 \right\} \; \E_\Q \!\left[ \eta_1 \middle\vert \sT_\eta \right] \bigg\vert \\
 & = & \frac{1}{\E_\P \!\left\| \xi \right\|} \, \E_\Q \bigg\vert n^{-1} \sum_{j = 1}^n \eta_j - \one \!\left\{ \left\| \xi \right\|  > 0 \right\} \; \E_\Q \!\left[ \eta_1 \middle\vert \sT_\eta \right] \bigg\vert \\
& = & \frac{1}{\E_\P \!\left\| \xi \right\|} \, \E_\Q \bigg\vert n^{-1} \sum_{j = 1}^n \eta_j - \E_\Q \!\left[ \eta_1 \middle\vert \sT_\eta \right] \bigg\vert \,.
\end{eqnarray*}
The right-hand side converges to zero as $n \to \infty$.
\end{proof}

\section{Ergodic theorem}
\label{sec-erg}

To formulate our ergodic theorem for swap-invariant random measures, we need to introduce some notions.

\begin{definition}
Let $(S,\sS,\mu)$ be a measure space. An increasing sequence $A_n \in \sS$ with $\mu (A_n) < \infty$ for $n \geq 1$ and $\mu (A_n) \to \infty$ as $n \to \infty$ is called {\em $\mu$-sequence}. For a $\mu$-sequence $(A_n)_{n \geq 1}$ we write
\[
\Delta A_n = \left\{
\begin{array}{ll}
A_1 & \quad \mathrm{if} \;\, n = 1 \\[.5em]
A_n \!\setminus\! A_{n-1} & \quad \mathrm{if} \;\, n \geq 2 
\end{array}\right. \,.
\]
A $\mu$-sequence $(A_n)_{n \geq 1}$ is called {\em $\mu$-sequence with constant increments} if $\mu (\Delta A_n) = c$ for all $n \geq 1$ and some $c \in (0, \infty)$. Moreover, given a random measure $\xi$ on~$S$, a $\mu$-sequence $(A_n)_{n \geq 1}$ is called {\em $\xi$-integrable} if $\E \, \xi (A_n) < \infty$ for~$n \geq 1$.
\end{definition} 

Clearly, if the conditions of Theorem~\ref{siintmeas} are satisfied, each $\mu$-sequence is $\xi$-integrable.

For a fixed $\mu$-sequence with constant increments it is straightforward to derive an ergodic theorem by applying the result for swap-invariant sequences as follows.

\begin{proposition}
\label{lemma-erg}
Let $(S,\sS,\mu)$ be a measure space, $\xi$ a $\mu$-swap-invariant random measure on~$S$, and $(A_n)_{n \geq 1}$ a $\xi$-integrable $\mu$-sequence with constant increments. Then there exists an integrable random variable~$X$ such that $\xi (A_n) / \mu (A_n) \to X$ almost surely as $n \to \infty$.
\end{proposition}

\begin{proof}
We have $\E \, \big| \sum_{j = 1}^n u_j \, \xi \!\left( \Delta A_j \right)\!\big| = \E \, \big| \sum_{j = 1}^n u_j \, \xi \!\left( \Delta A_{\pi(j)} \right)\!\big|$ for $n \geq 1$, $u \in \R^n$, and $\pi \in \perm(n)$. Therefore the random sequence $(\xi (\Delta A_j))_{j \geq 1}$ is almost surely equal to a swap-invariant sequence of integrable $\R_+$-valued random variables. By~\cite[Theorem 17]{mss:14}, there exists an integrable random variable~$X$ such that
\[
\frac{\xi (A_n)}{\mu (A_n)} = \frac{1}{\mu(A_1)} \frac{1}{n} \sum_{j = 1}^n \xi \!\left( \Delta A_j \right) \to X \quad \mbox{a.s.\ as $n \to \infty$} \,. \qedhere
\]
\end{proof}

\begin{example}
\label{ex-poisson}
Let $(S,\sS,\mu)$ be a measure space where $\mu$ is $\sigma$-finite and $\mu(S) = \infty$. Further let $\eta$ be a Poisson process on~$S$ with intensity measure~$\mu$. Since $\eta$ is $\mu$-exchangeable, it is $\mu$-swap-invariant. Now let $(A_n)_{n \geq 1}$ be a $\mu$-sequence with constant increments and define $c = \mu(A_1)$. Then $(\eta (\Delta A_j))_{j \geq 1}$ is an i.i.d.\ sequence of $\R_+$-valued integrable random variables. Thus we obtain almost surely and in~$L^1$ as $n \to \infty$
\[
\frac{\eta (A_n)}{\mu (A_n)} = \frac{1}{c} \frac{1}{n} \sum_{j = 1}^n \eta \!\left( \Delta A_j \right) \, \to \, \frac{1}{c} \, \E \, \eta (A_1) = \frac{1}{c} \, \mu(A_1) = 1 \,.
\]
\end{example}

\begin{example}
\label{ex-swap-rm-erg}
In Example~\ref{ex-swap-rm} let $\mu(S) = \infty$. Further let $(A_n)_{n \geq 1}$ be a $\mu$-sequence with constant increments and define $c = \mu(A_1)$. Then as in Example~\ref{ex-poisson}, we have $\eta (A_n) / \mu (A_n) \to 1$ as $n \to \infty$ $\Q$-almost surely and in $L^1(\Q)$. It follows that $\xi (A_n) / \mu (A_n) \to X\,$  $\Q$-almost surely. Since $\P$ and $\Q$ are equivalent, this convergence holds also $\P$-almost surely. It can be shown by direct computation that the convergence is also in~$L^1(\P)$.
\end{example}

We now show that the ergodic limit also exists if the increments are not necessarily constant and that the limit is unique under certain assumptions. This allows us to perform a change of the probability measure and to construct a random measure that is $\mu$-exchangeable under the new probability measure in order to obtain~\eqref{intro-repr}.

\begin{theorem}
\label{th-erg-conv}
Let $(S,\sS, \mu)$ be an atomless measure space and $\xi$ a random measure on~$S$ that is $\mu$-swap-invariant under a probability measure~$\P$.
\begin{benumber}
\item \label{erg-conv} For each $\xi$-integrable $\mu$-sequence $(A_n)_{n \geq 1}$ there exists an integrable random variable~$X$ such that \mbox{$\xi (A_n) / \mu (A_n) \to X$} almost surely as $n \to \infty$.
\item \label{erg-change} Assume that the measures $\mu$ and $\E_\P \xi$ are $\sigma$-finite. Further assume that there exists a $\mu$-sequence $(A_n)_{n \geq 1}$ with constant increments and limit~$A$ such that $\mu (S \setminus A) = \infty$, \mbox{$\xi (A_n) / \mu (A_n) \to X$} in~$L^1(\P)$, and $\E_\P X > 0$. Then \mbox{$\xi (B_n) / \mu (B_n) \to X$} in~$L^1(\P)$ for each $\mu$-sequence $(B_n)_{n \geq 1}$\,.
\item \label{erg-qeta} Under the same conditions as in~\eqref{erg-change}, there exists a random measure $\eta$ that is $\mu$-exchangeable under the probability measure $\Q$ defined by
\begin{equation} \label{equ-q-meas}
\frac{\mrmdd\Q}{\mrmdd\P} = \frac{X}{\E_\P X }
\end{equation}
such that $\xi = X \eta\,$ $\P$-almost surely.
\end{benumber}
\end{theorem}

Recall that under the conditions of part~\eqref{erg-change}, each $\mu$-sequence is $\xi$-integrable.

\begin{example}
Let $(S,\sS)$ be a Borel space and $\mu$ an atomless and $\sigma$-finite measure on~$S$. Further let $\xi$ be an almost surely diffuse random measure on~$S$ such that $\E \, \xi$ is $\sigma$-finite and $\xi$ is $\mu$-swap-invariant. It follows from Theorem~\ref{theo-diffuse-rm} that $\xi = \alpha \mu$ almost surely for some $\R_+$-valued integrable random variable~$\alpha$. For each $\mu$-sequence $(A_n)_{n \geq 1}$ we clearly have \mbox{$\xi (A_n) / \mu (A_n) \to \alpha$} almost surely and in~$L^1$.
\end{example}

\begin{example}
In Example~\ref{ex-swap-rm} let $\mu(S) = \infty$. Since $\xi$ satisfies the conditions of Theorem~\ref{th-erg-conv}~\eqref{erg-change} (see Examples \ref{eq-swap-rm-int} and~\ref{ex-swap-rm-erg}), it follows that $\xi (B_n) / \mu (B_n) \to X\,$ $\P$-almost surely and in~$L^1(\P)$ for each $\mu$-sequence $(B_n)_{n \geq 1}$\,.
\end{example}

In the proof of Theorem~\ref{th-erg-conv} we make use of the following notion.

\begin{definition}
Let $(S,\sS,\mu)$ be a measure space and $(A_n)_{n \geq 1}$ a $\mu$-sequence. A $\mu$-sequence $(C_n)_{n \geq 1}$ with constant increments is called {\em compatible sequence with constant increments (CSCI) of~$(A_n)$} if there exists a $\mu$-sequence $(B_n)_{n \geq 1}$ such that $(A_n)$ and $(C_n)$ are subsequences of~$(B_n)$.
\end{definition}

Obviously in this definition the sequences $(A_n)$, $(B_n)$, and $(C_n)$ have the same limit set. We now prove two lemmas on which Theorem~\ref{th-erg-conv} is based.

\begin{lemma}
\label{csci-prop}
Let $(S,\sS,\mu)$ be a measure space where $\mu$ is atomless and $\mu (S) = \infty$, $(A_n)_{n \geq 1}$ a $\mu$-sequence, $\xi$ a random measure on~$S$, and $X$ a random variable.
\begin{benumber}
\item \label{csci-prop-allc} For each $c \in (0,\infty)$, there exists a CSCI $(C_n)$ of~$(A_n)$ with $\mu (C_1) = c$.
\item  \label{csci-prop-as} If \,\mbox{$\xi (C_n) / \mu (C_n) \to X$} almost surely as $n \to \infty$ for some CSCI $(C_n)$ of~$(A_n)$, then also \,\mbox{$\xi (A_n) / \mu (A_n) \to X$} almost surely.
\item  \label{csci-prop-l1} If \,\mbox{$\xi (C_n) / \mu (C_n) \to X$} in~$L^1$ as $n \to \infty$ for some CSCI $(C_n)$ of~$(A_n)$, then also \,\mbox{$\xi (A_n) / \mu (A_n) \to X$} in~$L^1$.
\end{benumber}
\end{lemma}

\begin{proof}
\eqref{csci-prop-allc} is clear because $\mu$ is atomless. In order to prove \eqref{csci-prop-as} and~\eqref{csci-prop-l1} let $(C_n)$ be a CSCI of~$(A_n)$. Define $c = \mu (C_1)$ and $m_k = \min \left\{ m \geq 1 \, ; \, A_k \subset C_m \right\}$ for $k \geq 1$. It follows that $1 \leq m_1 \leq m_2 \leq \ldots$, and $m_k \to \infty$ as $k \to \infty$. For large $k$ we have $m_k \geq 2$ and $C_{m(k) - 1} \subset A_k \subset C_{m(k)}$ where the first inclusion is strict and the second may not. Hence, for large~$k$,
\[
\mu (C_{m(k) - 1}) \leq \mu (A_k) \leq \mu (C_{m(k)}), \quad\quad \xi (C_{m(k) - 1}) \leq \xi (A_k) \leq \xi (C_{m(k)}),
\]
and therefore
\[
\left( \frac{m_k - 1}{m_k} \right) \frac{\xi (C_{m(k) - 1})}{\mu (C_{m(k) - 1})} = \frac{\xi (C_{m(k) - 1})}{\mu (C_{m(k)})} \leq \frac{\xi(A_k)}{\mu(A_k)} \leq \frac{\xi (C_{m(k)})}{\mu (C_{m(k) - 1})} = \left( \frac{m_k}{m_k - 1} \right) \frac{\xi (C_{m(k)})}{\mu (C_{m(k)})} \,.
\]
Thus if $\xi (C_n) / \mu (C_n) \to X$ almost surely, then $\xi (A_n) / \mu (A_n) \to X$ almost surely as $n \to \infty$. This proves~\eqref{csci-prop-as}. From the same estimate we obtain, for large~$k$:
\begin{align*}
&\E \left\vert \frac{\xi (A_k)}{\mu (A_k)} - X \right\vert \, \leq \, \E \left\vert \frac{\xi (A_k)}{\mu (A_k)} - \left( \frac{m_k}{m_k - 1} \right) \frac{\xi (C_{m(k)})}{\mu (C_{m(k)})} \right\vert + \E \left\vert  \left( \frac{m_k}{m_k - 1} \right) \frac{\xi (C_{m(k)})}{\mu (C_{m(k)})} - X \right\vert \\[1em]
&\leq\, \E \left\vert \left( \frac{m_k}{m_k - 1} \right) \frac{\xi (C_{m(k)})}{\mu (C_{m(k)})} - \left( \frac{m_k - 1}{m_k} \right) \frac{\xi (C_{m(k) - 1})}{\mu (C_{m(k) - 1})} \right\vert + \E \left\vert \left( \frac{m_k}{m_k - 1} \right) \frac{\xi (C_{m(k)})}{\mu (C_{m(k)})} - X \right\vert \\[1em]
&\leq\, 2 \, \E \left\vert \left( \frac{m_k}{m_k - 1} \right) \frac{\xi (C_{m(k)})}{\mu (C_{m(k)})} - X \right\vert + \E \left\vert X - \left( \frac{m_k - 1}{m_k} \right) \frac{\xi (C_{m(k) - 1})}{\mu (C_{m(k) - 1})} \right\vert \,.
\end{align*}
The right-hand side converges to zero as $k \to \infty$ if $\xi (C_n) / \mu (C_n) \to X$ in~$L^1$ as $n \to \infty$. This proves~\eqref{csci-prop-l1}.
\end{proof}

\begin{lemma}
\label{erg-ex-prop}
Let $(S,\sS,\mu)$ be a measure space with $\mu (S) = \infty$, and $\xi$ a random measure on~$S$ that is $\mu$-swap-invariant under a probability measure~$\P$. Further let $(A_n)_{n \geq 1}$ be a $\xi$-integrable $\mu$-sequence with constant increments and limit~$A$ such that $\xi (A_n) / \mu (A_n) \to X\,$ $\P$-almost surely and in~$L^1 (\P)$ as $n \to \infty$ for some random variable~$X$ with $\E_\P X > 0$. Define the random measure $\eta$ by
\begin{equation} \label{equ-eta-meas}
\eta = \left\{ \begin{array}{ll} \xi / X & \quad \mathrm{on} \;\left\{X > 0 \right\} \\[1em]
 0 & \quad \mathrm{on} \;\left\{X = 0 \right\} \end{array}
\right.
\end{equation}
and $\Q$ by~\eqref{equ-q-meas}.
\begin{benumber}
\item \label{erg-ex-prop-1} The sequence $(\eta (\Delta A_n))_{n \geq 1}$ is exchangeable under~$\Q$.
\item \label{erg-ex-prop-1a} $\xi(A_n) = 0\,$ $\P$-almost surely on $\left\{ X = 0 \right\}$ for $n \geq 1$.
\item \label{erg-ex-prop-2} For each $m \geq 1$ and disjoint measurable sets $(B_j)_{1 \leq j \leq m}$ with $\mu (B_j) = \mu (A_1)$, $B_j \cap A = \emptyset$, and $\E_\P \xi (B_j) < \infty$, we have, under~$\Q$,
\[
\big( \eta (B_1), \ldots, \eta (B_m) \big) \,\deq\, \big( \eta (\Delta A_1), \ldots, \eta (\Delta A_m) \big) \,.
\]
\end{benumber}
\end{lemma}

\begin{proof}
The sequence $(\xi(\Delta A_n))_{n \geq 1}$ is $\P$-almost surely equal to a swap-invariant sequence of $\R_+$-valued random variables, that satisfies the conditions of Corollary~\ref{cor-nonneg-seq}; this proves~\eqref{erg-ex-prop-1} and~\eqref{erg-ex-prop-1a}.

Now choose a sequence $\zeta$ of random variables in $\R_+$ that are swap-invariant under~$\P$ such that $\P$-almost surely $\zeta_j = \xi (B_j)$ for $1 \leq j \leq m$ and $\zeta_{j+m} = \xi (\Delta A_j)$ for $j \geq 1$. We have $n^{-1} \sum_{j = 1}^n \zeta_j \to \mu(A_1) X\,$ $\P$-almost surely and in~$L^1 (\P)$. Another application of Corollary~\ref{cor-nonneg-seq} shows~\eqref{erg-ex-prop-2}.
\end{proof} 

\begin{proof}[Proof of Theorem~\ref{th-erg-conv}]
We first show~\eqref{erg-conv}. By Lemma~\ref{csci-prop}~\eqref{csci-prop-allc} there is a CSCI $(C_n)$ of~$(A_n)$ with $\mu (C_1) = 1$. In particular $(C_n)$ is a $\xi$-integrable $\mu$-sequence, so by Proposition~\ref{lemma-erg} there is an integrable random variable~$X$ such that $\xi (C_n) / \mu (C_n) \to X\,$ $\P$-almost surely. By Lemma~\ref{csci-prop}~\eqref{csci-prop-as}, we know that also $\xi (A_n) / \mu (A_n) \to X\,$ $\P$-almost surely.

In order to show~\eqref{erg-change}, assume that $(A_n)$, $A$, and~$X$ have the stated properties, and let $c = \mu (A_1)$. Now let $(B_n)$ be another $\mu$-sequence, say with limit~$B$. First we assume that $(B_n)$ has constant increments with $\mu (B_1) = c$, and that $A \cap B = \emptyset$. Clearly $\E_\P \xi (B_n) < \infty$ for all~$n$. By swap-invariance
\begin{eqnarray*}
\lefteqn{\E_\P \bigg\vert n^{-1} \sum_{j = 1}^n \xi (\Delta B_j) - m^{-1} \sum_{k = 1}^m \xi (\Delta A_k) \bigg\vert} \\
 & = & \E_\P \bigg\vert n^{-1} \sum_{j = 1}^n \xi (\Delta A_j) - m^{-1} \bigg( \sum_{k = 1}^n \xi (\Delta B_k) + \sum_{k = n + 1}^m \xi (\Delta A_k)\bigg) \bigg\vert
\end{eqnarray*}
for $m > n \geq 1$. Letting $m \to \infty$, it follows that
\[
\E_\P \bigg\vert n^{-1} \sum_{j = 1}^n \xi (\Delta B_j) - c X \bigg\vert \, = \, \E_\P \bigg\vert n^{-1} \sum_{j = 1}^n \xi (\Delta A_j) - c X \bigg\vert \,.
\]
Letting $n \to \infty$ shows that $\xi (B_n) / \mu (B_n) \to X\,$ in~$L^1(\P)$, and \eqref{erg-conv} implies that this convergence is also $\P$-almost surely. Now let $(B_n)$ be an arbitrary $\mu$-sequence with limit~$B$, i.e.\ we may have $A \cap B \neq \emptyset$. We may choose a $\mu$-sequence with constant increments $(E_n)$ with limit~$E$ such that $\mu (E_1) = c$ and $A \cap E = \emptyset$. By the first part of the proof it follows that $\xi (E_n) / \mu (E_n) \to X$ $\P$-almost surely and in~$L^1(\P)$. Now we distinguish the cases $\mu (A \cap B) < \infty$ and $\mu (A \cap B) = \infty$. In the first case we choose another $\mu$-sequence with constant increments $(F_n)$ with limit~$F$ such that $\mu (F_1) = c$ and $F \subset A \setminus B$. Let $(C_n)$ be a CSCI of~$(B_n)$ with $\mu (C_1) = c$, which exists by Lemma~\ref{csci-prop}~\eqref{csci-prop-allc}. From the convergence of $\xi (E_n) / \mu (E_n)$ we consecutively conclude that the same convergence holds for $(F_n)$ and~$(C_n)$. Finally $\xi (B_n) / \mu (B_n) \to X\,$ $\P$-almost surely and in~$L^1(\P)$ by Lemma~\ref{csci-prop} \eqref{csci-prop-as} and~\eqref{csci-prop-l1}. In the second case, $\mu (A \cap B) = \infty$, we may choose a CSCI $(C_n)$ of~$(B_n)$ with $\mu (C_1) = 2c$, and two $\mu$-sequences with constant increments $(C_n^i)$ $(i = 1, 2)$ such that, for $n \geq 1$,
\[
C_n = C_n^1 \cup C_n^2 \, , \quad\quad C_n^1 \cap C_n^2 = \emptyset, \quad\quad \mu (\Delta C_n^1) = \mu (\Delta C_n^2) = c \,.
\]
For $i \in \left\{ 1, 2 \right\}$ let $C^i$ be the limit set of~$(C_n^i)$. Without loss of generality we may assume that $\mu (A \cap C^2) = \infty$. Then there is a \mbox{$\mu$-sequence} with constant increments $(F_n)$ and limit $F$ such that $\mu (F_1) = c$ and $F \subset A \cap C^2$. By the first part of the proof we consecutively conclude that $\xi (F_n) / \mu (F_n)$, $\xi (C_n^1) / \mu (C_n^1)$, and $\xi (C_n^2) / \mu (C_n^2)$ converge to~$X$ $\P$-almost surely and in~$L^1(\P)$. It follows that
\[
\frac{\xi (C_n)}{\mu (C_n)} = \frac{1}{2} \left( \frac{\xi (C_n^1)}{c \, n} + \frac{\xi (C_n^2)}{c \, n} \right) = \frac{1}{2} \left( \frac{\xi (C_n^1)}{\mu (C_n^1)} + \frac{\xi (C_n^2)}{\mu (C_n^2)} \right) \to X \quad \mbox{$\P$-a.s.\ and in~$L^1(\P)$} \,.
\]
Finally we find that $\xi (B_n) / \mu (B_n) \to X\,$ $\P$-almost surely and in~$L^1(\P)$ by Lemma~\ref{csci-prop} \eqref{csci-prop-as} and~\eqref{csci-prop-l1}. This shows statement~\eqref{erg-change}.

In order to prove \eqref{erg-qeta} first note that for any $D \in \sS$ with $\mu(D) < \infty$ we may choose a $\mu$-sequence with constant increments $(E_n)$ such that $E_1 = D$. Applying Lemma~\ref{erg-ex-prop}~\eqref{erg-ex-prop-1a} to $(E_n)$ gives $\xi(D) = 0\,$ $\P$-almost surely on $\left\{ X = 0 \right\}$. Now define $\eta$ as in~\eqref{equ-eta-meas}. It remains to show that $\eta$ is $\mu$-exchangeable under~$\Q$. Let $d \in (0,\infty)$, $m \geq 1$, and $(B_j)_{1 \leq j \leq m}$ disjoint measurable sets with $\mu (B_j) = d$. Define $B = \bigcup_{j = 1}^m B_j$\,. Since $\mu (S \setminus B) = \infty$, we may choose a $\mu$-sequence $(C_n)$ with constant increments and limit in~\mbox{$S \setminus B$} such that $\mu (C_1) = d$. Statement~\eqref{erg-change} implies that \mbox{$\xi (C_n) / \mu (C_n) \to X\,$} $\P$-almost surely and in~$L^1(\P)$. Applying Lemma~\ref{erg-ex-prop} to the sequence $(C_n)$ and sets $(B_j)_{1 \leq j \leq m}$\,, we obtain that, under~$\Q$,
\[
\big( \eta (B_1), \ldots, \eta (B_m) \big) \,\deq\, \big( \eta (\Delta C_1), \ldots, \eta (\Delta C_m) \big) \,.
\]
Now let $\pi \in \perm(m)$. Deriving the same relation for the permuted sets, we get, under~$\Q$,
\[
\big( \eta (B_1), \ldots, \eta (B_m) \big) \,\deq\, \left( \eta \!\left( B_{\pi(1)} \right), \ldots, \eta \!\left( B_{\pi(m)} \right) \right) .
\]
By Lemma~\ref{exch-random}~\eqref{exch-random-perm} this shows that $\eta$ is $\mu$-exchangeable under~$\Q$.
\end{proof}

Note that if $X$ in Theorem~\ref{th-erg-conv}~\eqref{erg-change} is almost surely constant, then $\xi$ is $\mu$-exchangeable under~$\P$, similarly to \cite[Corollary 24]{mss:14} for random sequences. For easier comparison with existing results we give a variant of Theorem~\ref{th-erg-conv} for $\mu$-exchangeable random measures.

\begin{theorem}
\label{th-erg-conv-exch}
Let $(S,\sS, \mu)$ be an atomless measure space and $\xi$ a $\mu$-exchangeable random measure on~$S$.
\begin{benumber}
\item \label{erg-conv-exch} For each $\xi$-integrable $\mu$-sequence $(A_n)_{n \geq 1}$ there exists a random variable~$X$ such that \mbox{$\xi (A_n) / \mu (A_n) \to X$} almost surely and in $L^1$ as $n \to \infty$.
\item \label{erg-change-exch} Assume that $\mu$ is $\sigma$-finite and that there is $C \in \sS$ with $\mu(C) \in (0,\infty)$ and \mbox{$\E \, \xi (C) < \infty$}. Then the limit in~\eqref{erg-conv-exch} is unique for all $\mu$-sequences.
\end{benumber}
\end{theorem}

\begin{proof}
Statement~\eqref{erg-conv-exch} is a consequence of Lemma~\ref{csci-prop} and Birkhoff's ergodic theorem, see e.g.~\cite[Theorem 10.6]{kallenberg:fou}. In order to show~\eqref{erg-change-exch} we may assume that $\mu(S) = \infty$. Note that Corollary~\ref{cor-int-exch} applies, so that $\E \, \xi$ is $\sigma$-finite and each $\mu$-sequence is $\xi$-integrable. Clearly there exists a $\mu$-sequence $(A_n)_{n \geq 1}$ with unit increments and limit~$A$ such that $\mu(S \setminus A) = \infty$. By~\eqref{erg-conv-exch} there exists a random variable~$X$ such that \mbox{$\xi (A_n) / \mu (A_n) \to X$} almost surely and in $L^1$. Now let $(B_n)_{n \geq 1}$ be another $\mu$-sequence, say with limit~$B$. We may assume that $(B_n)$ has unit increments and that $A \cap B = \emptyset$. The general case is then proven as in Theorem~\ref{th-erg-conv}. Define random sequences $\xi_j = \xi (\Delta A_j)$ and $\xi^\ast_j = \xi (\Delta B_j)$ for $j \geq 1$, and $g(x) = \limsup_{n \to \infty} n^{-1} \sum_{j = 1}^n x_j\,$ for $x \in \overline{\R}^{\, \infty}_+$. By the $\mu$-exchangeability of~$\xi$
\[
\left( \xi_1\,, \ldots, \xi_m\,, \xi_{m+1}\,, \ldots \right) \,\deq\, \left( \xi^\ast_1\,, \ldots, \xi^\ast_m\,, \xi_{m+1}\,, \xi_{m+2}\,, \ldots \right),
\]
for $m \geq 1$, and consequently
\begin{eqnarray*}
\big( \xi_1\,, \ldots, \xi_m\,, g(\xi) \big) & \deq & \big( \xi^\ast_1\,, \ldots, \xi^\ast_m\,, g( \xi^\ast_1\,, \ldots, \xi^\ast_m\,, \xi_{m+1}\,, \xi_{m+2}\,, \ldots) \big)\\
 & = & \big( \xi^\ast_1\,, \ldots, \xi^\ast_m\,, g(\xi) \big) \quad \mbox{a.s.}
\end{eqnarray*}
It follows that $(\xi, g(\xi)) \deq (\xi^\ast, g(\xi))$. Now \cite[Corollary 6.11]{kallenberg:fou} implies that $g(\xi) = g(\xi^\ast)$ almost surely.
\end{proof}

\begin{example}
Let $\xi$ be a $\lambda$-exchangeable random measure on~$\R^d$ with \mbox{$\E \, \xi ( \left[0,1\right]^d) < \infty$} where $\lambda$ denotes the Lebesgue measure. It follows that $\xi$ is stationary (cf.~\cite[p.~189]{kallenberg:fou}). Hence the ergodic theorem \cite[Corollary 10.19]{kallenberg:fou}, which is based on~\cite{ngz:79}, implies that $\xi(A_n) / \lambda(A_n) \to X$  almost surely and in $L^1$ as $n \to \infty$ for a certain subclass of $\lambda$-sequences, namely all sequences of increasing bounded convex Borel sets $(A_n)_{n \geq 1}$ such that the inner radius $r(A_n) \to \infty$. Theorem~\ref{th-erg-conv-exch} generalizes existence and uniqueness to all $\lambda$-sequences.
\end{example}

\section*{Acknowledgement}

The author is very grateful to Ilya Molchanov for guidance through the topic and numerous fruitful discussions. This work was supported by Swiss National Science Foundation Grant 200021-153597.

\bibliographystyle{plain}
\bibliography{swap-invariance}

\end{document}